\documentclass[8pt]{article}

\usepackage[utf8]{inputenc}
\usepackage[lined]{algorithm2e}
\usepackage{minitoc}
\usepackage{mathtools}

\usepackage{fancyvrb}
\usepackage{color}
\usepackage{multicol}

\usepackage{datetime}
\usepackage[left=3cm,right=3cm]{geometry}
\usepackage{soul}
\usepackage{cancel}
\usepackage[backend=biber,style=numeric,natbib=true,giveninits=true,url=false,doi=false,isbn=false,maxnames=5]{biblatex}
\usepackage[pdfauthor={Andres A Leon Baldelli},pdftitle={},pdfsubject={},pdfkeywords={},pdfproducer={Latex with hyperref},pdfcreator={pdflatex}]{hyperref}
\usepackage[english]{babel}

\usepackage{amsthm}
\usepackage{amsmath}
\usepackage{amssymb}
\usepackage{bbm}

\newtheorem{definition}{Definition}
\newtheorem{theorem}{Theorem}
\newtheorem{Remark}{Remark}\newtheorem{proposition}{Proposition}
\newtheorem{corollary}{Corollary}[theorem]
\newtheorem{lemma}{Lemma}
\usepackage{latexsym}

\usepackage{boxedminipage}

\usepackage{listings}

\usepackage{minitoc}

\usepackage{ifpdf}

\newenvironment{system}{\left\lbrace\begin{array}{@{}l@{}}}{\end{array}\right.}

\usepackage{changes}
\definechangesauthor[color=orange]{ALB}

\newcommand{\e}{\varepsilon}
\newcommand{\ue}{u^\e}

\renewcommand{\div}{\operatorname{div}}

\renewcommand{\theta}{\vartheta} 
\renewcommand{\phi}{\varphi}

\DeclareMathOperator{\tr}{tr}

\providecommand{\tr}{\mathrm{tr}}

\providecommand{\nltwo}[2][]{|| #1 ||_{#2}}

\newcommand{\N}{{\mathbb{N}}}
\newcommand{\R}{{\mathbb{R}}}

\providecommand{\pt}{\partial_3}

\providecommand{\pa}{\partial_\alpha}

\providecommand{\pab}{\partial_{\alpha\beta}}

\providecommand{\ua}{u_{\alpha}}
\providecommand{\va}{v{_\alpha}}
\providecommand{\vt}{v{_3}}

\providecommand{\eab}{e_{\alpha\beta}}
\providecommand{\eaa}{e_{\alpha\alpha}}

\providecommand{\ett}{e_{33}}
\providecommand{\eat}{e_{\alpha 3}}

\let\wto\rightharpoonup

\renewcommand{\O}{\Omega}
\renewcommand{\o}{\omega}

 \let\e\varepsilon

\providecommand{\eab}{e_{\alpha\beta}}
\providecommand{\eaa}{e_{\alpha\alpha}}

\providecommand{\eat}{e_{\alpha 3}}
\providecommand{\ett}{e_{33}}
\providecommand{\eat}{e_{\alpha 3}}
\providecommand{\eit}{e_{i 3}}

\providecommand{\Qab}{Q_{\alpha\beta}}
\providecommand{\Qtt}{Q_{33}}
\providecommand{\Qat}{Q_{\alpha3}}

\providecommand{\ue}{u^\e}
\providecommand{\uea}{\ue_\alpha}
\providecommand{\uet}{\ue_3}
\providecommand{\ut}{{u_3}} 

\providecommand{\ve}{v^\e}

\providecommand{\Qtt}{Q_{33}}

\newcommand{\Of}{{\O_f}}
\newcommand{\Ob}{{\O_b}}
\newcommand{\Oef}{{\O^\e_f}}
\newcommand{\Oeb}{{\O^\e_b}}

 \renewcommand{\Ob}{{\Omega_{b}}}
\renewcommand{\Oeb}{{\Omega^\e_{b}}}

\usepackage{changes}

\usepackage{booktabs}

\usepackage{graphicx}
\title{Gamma-convergence results for nematic elastomer bilayers: relaxation and actuation}
\author{Pierluigi Cesana, Andrés A. León Baldelli}

\begin{document}

\ifpdf
\DeclareGraphicsExtensions{.pdf, .jpg, .tif}
\else
\DeclareGraphicsExtensions{.eps, .jpg}
\fi

\maketitle

\begin{abstract}
We compute effective energies of thin bilayer structures composed of soft nematic elastic liquid crystals in various geometrical regimes and functional configurations.
Our focus is on elastic foundations composed {of} an isotropic layer attached to a nematic substrate where order-strain interaction results in complex opto-mechanical instabilities  activated via coupling through the common interface.
We compute Gamma-limits  for vanishing thickness in two main scaling regimes respectively exhibiting spontaneous stress relaxation and shape-morphing, allowing in both cases out-of-plane displacements. This extends the plane strain modelling of \cite{cesana2018variational}, and shows the asymptotic emergence of fully coupled macroscopic active-nematic foundations.
Subsequently, we focus on actuation and compute asymptotic configurations of an active plate on nematic foundation interacting with an applied electric field. From the analytical standpoint, the presence of an electric field and its associated electrostatic work turns the total energy  non-convex and non-coercive.
We show that equilibrium solutions are min-max points of the system, that min-maximising sequences pass to the limit and, that the limit system can exert mechanical work under applied electric fields.
 \end{abstract}
 

\tableofcontents

\section{Introduction}

are classes of soft shape-memory alloys where order states and optical instabilities can be triggered, tuned, or suppressed by means of mechanical deformations and electrostatic fields. 
NLCEs, typically synthesised as thin strips, appear in the form of gels or rubbery solids. 
Structurally, they are constituted by polymeric chains which act as the 
backbone of the material, to which attach molecules of a nematic liquid crystal, the optically active units. 
Liquid crystal molecules have a two-fold response to stimuli.
Indeed, they re-orient 
{parallel to a common direction (identified by a unit vector $n$ called director)}
as a consequence of elastic deformations and 
stretches dictated by \emph{internal energy} minimisation and they rotate collectively parallel to electric or magnetic forces,
activated by \textit{external fields}.

We are interested in testing and analysing the interaction between elastic, optic, and electrostatic forces (characterising the \textit{material} behaviour) and  geometric constraints (which cause \textit{structural} instabilities) in two distinct physical regimes relevant for technological applications.

Continuum modelling of nematic elastomers based on the Gaussian approximation of the distribution of polymer chains traces back to the work of \cite{BTW93}, \cite{WT03}. Subsequently, such mechanical models have been cast in a variational framework within the context of non-convex minimisation in nonlinear elasticity, see \cite{DeS99}, \cite{desimone2002macroscopic}.
{In the present contribution we focus on linearised elasticity. }
Despite intrinsic limitations of an approach with infinitesimal displacements, a linearised model is particularly suitable for treating the coupling of multiphysical phenomena, by superposition.
We refer to \cite{DST09} and \cite{ADS11} for a discussion on the relationship between the linearised model adopted here and
the classical nonlinear theory.
{In our setting, the twofold multiphyisical interaction pertains to the interplay between nematic ordering and elasticity  and to the opto-electric interaction. While the former effect is caused by the coupling of polymeric chains with LC optic states, the latter is the connection between the liquid crystal and the dielectric vector field.}

We consider an energy model introduced in \cite{cesana2009strain-order} for the description of equilibrium states of NLCEs under an electric field  to study the asymptotic behaviour of sequences of functionals for bi-layer structures where a NLCE membrane sustains a stiff and thin isotropic film.

Two non-dimensional quantities (interpreted as length scale ratios) collapse several material and geometric parameters identifying two opposite phenomenological regimes: that of thin films of large planar extent (which we call the ``large bodies''), and that of thin films of small area (referred to as ``thin particles''). We follow  the terminology and modelling philosophy introduced in~\cite{de-simone1993energy}.

A number of contributions have appeared in recent years on the mathematical modelling of {effective} thin nematic elastomer structures in various geometries and regimes.
Mechanical relaxation and formation of microstructure are analysed in semi-soft planar membranes in~\cite{Conti2002PRE},\cite{conti2002soft}, and in relation to opto-mechanical wrinkling in~\cite{cesana2015effective}.
In \cite{ADS20} the rich order-stretch interaction is studied for nonlinear plate models of thin nematic monolayers, showing that the asymptotic elastic behaviour dramatically depends on the morphology of nematic textures.
{In the same spirit,} one-dimensional finite elasticity models of NLCE ribbons are derived in~\cite{ADSK17} and~\cite{ADS17}, 
showing that imprinted LC director arrangements lead to nontrivial shape designs through  spontaneous  activation of flexural and torsional stretches.

The main contribution in this paper is the derivation, the analysis, and the computation of effective two-dimensional plate models for multilayer structures comprising active soft nematic plates, in the regimes of spontaneous relaxation and shape-morphing actuation.
In both limits the underlying elastic behaviour is represented by an effective linear plate energy of Kirchhoff-Love type, enriched by an additional nonlinear active foundation term which we explicitly characterise.
The latter contribution, which is a consequence of the nemato-elastic coupling in the active layer, emerges in the limiting structure thanks to kinematic compatibility at the interface.

The first regime (that of large bodies) entails the relaxation of elastic stresses by formation of optimal optic microstructure, with full or partial coupling between in-plane and out-of-plane deformations depending on the thickness scaling. This process ultimately characterises the mechanical behaviour of large thin elastic sheets.
This setting is explored in the first part of this paper, inspired by observations of pattern formation and opto-elastic relaxation in bilayer systems of nematic elastomers, see \cite{greco2013reversible}. There, a complex material-structure interaction is observed in thin membranes of NLCE in contact with an overlying isotropic film, resulting in formation of micro-wrinkles competing with visible shear-band microstructures of the optical axis.
The former manifestation is a typical structural instability, also observed in thin stretchable membranes~\cite{vandeparre2010hierarchical,bella2014wrinkles}, whereas the latter is a material instability observed in various classes of media, encompassing, e.g. solid elastic crystals(\cite{ball1989fine}) besides liquid crystal elastomers.
This paper complements and completes the analysis developed in~\cite{cesana2018variational} on planar-constrained nematic bilayers by exploring the full three-dimensional elastic model allowing for out-of plane elastic deformations.
In our main relaxation results,~{Theorems~\ref{thm:relaxcoupled} and~\ref{thm:relaxweaklycoupled}, } we compute effective energies of thin (fully coupled) and thick (weakly coupled) nematic foundations. 
In the  former  regime, we show that out-of-plane deformations can be activated by in-plane traction boundary data
as a by-product of material relaxation and nematic coupling, 
as well as that antiplanar optical states can be tuned through planar boundary conditions and viceversa.
This is a nontrivial coupling mechanism with potential technological applications in opto-mechanical sensing devices.

The latter regime (small particles), corresponds to the physical setting of small multilayer domains which can be actuated into nontrivial out-of-plane deformation modes by uniform fields, as we show in the second part of the paper.
Building upon 
the
relaxation result,
we describe this asymptotic behaviour, where the optic director is free to rotate and realign under the influence of applied external fields, albeit homogeneously.
These small domains may be regarded as elementary ``building blocks'' for more complex morphing shapes that can be assembled via actuation of the frozen director, in presence of suitable boundary conditions.
In the second part of this article we turn our attention on actuation, i.e. the capability of controlling the shape of a membrane thanks to the activation of liquid crystal molecules by external fields. Our investigation is inspired by a number of recent experimental realisations. In \cite{plucinsky2016programming}, \cite{plucinsky2018patterning} and~\cite{Kvari20}, design of complex shapes is performed via thermal actuation of heterogeneously patterned nematic elastomer films.
Using thin motion-controlled strips~\cite{DSGN15} conceptualises soft nematic elastomer robots;
in \cite{bai2020photomechanical} and \cite{korner2020a-nonlinear} opto-elasticity in nematic elastomers is investigated to show that soft NLCE robots can execute work cycles thanks to the cooperative interaction between light absorption and mechanical deformations.
We refer to \cite{white2015programmable} for a review focussed on liquid crystalline materials from the view point of thermal-photo-elastic coupling. 

From the mathematical perspective, we perform the exact computation of the $\Gamma$-limit of a family of energy functionals parametrised by the two scale parameters.
For the actuation problem, because of the presence of an energetic contribution due to external fields possibly unbounded below, we face the issue of non-convexity and non-coercivity of the total energy functional which thus lacks sequential lower semicontinuity.
The asymptotic analysis is nonetheless successful in 
showing that limit regime enjoys a saddle structure   by computing the exact expression of asymptotic Lagrangians.
In our main contribution (Theorems~\ref{thm:actuationcoupled} and~\ref{thm:actuationthick}) we demonstrate that the limit system can indeed transform and convert work produced by electrostatic forces into shape deformation, which is, to date, a challenge in soft robotics.
To illustrate our purpose, we numerically solve a simple actuation problem for membrane bending and show, as a mathematically relevant example, the  equilibrium configuration of a nematic bilayer {induced by an imprinted LC arrangement}.

The outline of the paper is as follows. After  presenting   the functional setting in the Introduction,  we discuss the kinematics and the mechanics of the problem (Section 2). 
Section 3 is devoted to the analysis of relaxation results for thin and large NLCE bilayers. 
{
Because these build substantially upon material produced in \cite{cesana2018variational}, we 
limit to the body of the article only essential proofs,
postponing mathematical details in Appendix~\ref{sec:appendix}.}
In Section 4 we analyse limit functionals for thin and small NLCE bi-layers. 
After characterising  the asymptotic behaviour of  saddle points of the energy functionals, as an example of our analytical work, we describe numerical calculations showing shape-actuation.

\subsection{Notation} \label{par:notation}

Throughout the paper, Greek indices run from $1$ to $2$ whereas Latin indices run from $1$ to $3$. The summation convention on repeated indices is assumed, unless explicitly stated. To highlight the dependency with respect to in-plane vs. out-of-plane coordinates, a prime sign indicates planar components of a vector, of a second order tensor, and of differential operators, 
as in
$v',B',$ and $\nabla'(\cdot)=\partial_\alpha (\cdot)$ respectively.
We use $\iota_\alpha$ to indicate unit vectors in the $x_1-x_2$ plane.
In order to distinguish homologue quantities defined in the two layers, we superpose a hat to those which refer to the nematic layer, as in $k, \hat k$ to indicate limit rescaled strains in the film and nematic layer, respectively.
The inner product is denoted by a dot.
In general (but with some exceptions, like $\nu$), material parameters or effective coefficients are indicated by sans serif letters, cf. Table~\ref{2008061744} for a collection of relevant parameters and physical constants.
With $u \otimes_s v$ we signify the symmetrised outer product $\frac{1}{2}(u\otimes v + u\otimes v) $ between vectors $u, v$ and by  $I$ the identity matrix in $\R^{n\times n}$.
Throughout the paper, $C$ stands for a generic constant which may change from line to line. Thickness averages are indicated by an overbar, as in $\bar v(x'):= 1/H \int v(x', x_3)dx_3$ where $H$ denotes the size of the (transverse) integration domain. We adopt standard notation for functional spaces, such as $L^2(\O, \R^n)$,
$L^2(\O, \R^{n\times n})$, and $H^1(\O, \R^n)$, $H^1(\O, \R^{n\times n})$, for the Lebesgue spaces of square integrable maps from $\O$ onto $\R^n$ and $\R^{n\times n}$, and the Sobolev space of square integrable maps with square integrable weak derivatives on $\O$. Concisely, we write $L^2(\O)$ and $H^1(\O)$ whenever $n=1$.
All $\e$-dependent quantities refer to the physical three-dimensional system, a thin bilayer structure whose thickness depends on $\e$. After introducing appropriate scalings for all material quantities and rescaling the physical domain we drop the $\e$-dependence.

\newcommand{\mmu}{}
\newcommand{\lamm}{{\frac{\nu}{1-2\nu}}}
\newcommand{\coefopt}{{\frac{\nu }{1-\nu}}}
\newcommand{\coeftransvopt}{{\frac{-\nu}{1-\nu}}}
\newcommand{\avguet}{\bar u^\e_3}
\newcommand{\avguea}{\bar u^\e_\alpha}

\newcommand{\Jepfunc}{\mathcal{J}^p_\e}
\newcommand{\Jz}{\mathcal{J}^0}
\newcommand{\Jm}{\mathcal{J}^-}

\newcommand{\Jepen}{J^p_\e}

\newcommand{\Jez}{\mathcal J^0_\e}
\newcommand{\Jep}{\mathcal J^p_\e}
\newcommand{\Jepj}{\mathcal J^p_{\e_j}}
\newcommand{\Jp}{\mathcal J^p}

\newcommand{\Fp}{{\mathcal F}^p}
\newcommand{\Fpe}{\mathcal{F}^p_\e}
\newcommand{\Fpej}{\mathcal{F}^p_{\e_j}}
\newcommand{\Fz}{\mathcal{ F}^0}
\newcommand{\Fm}{\mathcal{ F}^-}
\newcommand{\Fze}{\mathcal{ F}^0_\e}

\newcommand{\Jg}{J^0}
\newcommand{\Jmg}{J^-}

\newcommand{\ukl}{{u}}

\newcommand{\uklsad}{{\ukl^*}}
\newcommand{\Qsad}{{\overline{Q}^*}}
\newcommand{\phisad}{{\overline{\phi}_{\Qsad}}}
 \newcommand{\termA}{{\textcircled{\small a}}}
 \newcommand{\termB}{{\textcircled{\small b}}}
 \newcommand{\termC}{{\textcircled{\small c}}}
 \newcommand{\termD}{{\textcircled{\small d}}}
 \newcommand{\termE}{{\textcircled{\small e}}}
 \newcommand{\termEi}{{\textcircled{\small f}}}
 \newcommand{\termEii}{{\textcircled{\small g}}}
 \newcommand{\termEiii}{{\textcircled{\small h}}}

\section{Setting of the problem}

\paragraph{Domain.}
Let $\O^\e$ be a sufficiently smooth three-dimensional domain, constituted by the union of two thin layers: a linearly elastic film occupying $\Oef=\o\times(0, \e L)$ and a soft nematic elastomer occupying $\Oeb=\o\times (-\e^{p+1}L, 0]$  where $\e\ll 1$ is a small parameter.
The two layers are attached to a rigid substrate which imposes a hard condition of place, see Figure~\ref{2008062138}.
The basis of the cylindrical three-dimensional domain is $\o\subseteq \R^2$ with characteristic size $L>0$.
We focus on \emph{thin} limit systems as $\e\to 0$, by requiring that $p+1>0$.
\begin{figure}[tb]
	\centering  
	\includegraphics[height=4cm]{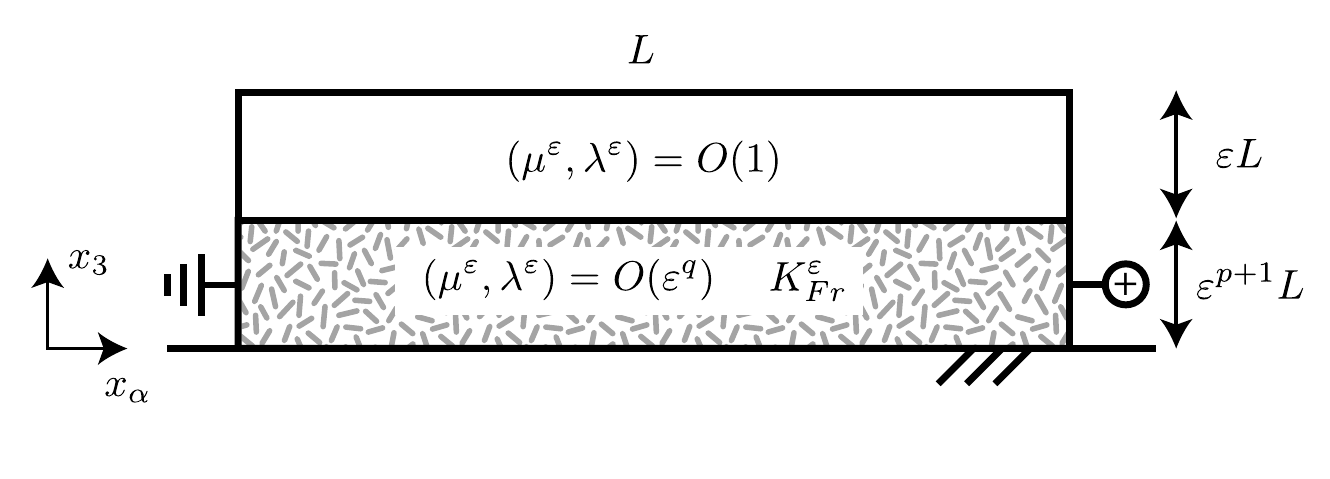}
\caption{Physical three-dimensional domain of a thin bilayer system consisting in one nematic elastomer layer supporting a stiff thin film. The system undergoes in-plane (membrane) and out-of-plane (bending) displacements, subject to mechanical volume and surface loads as well as electrical work inducing nematic reorientation. We distinguish two physically relevant regimes, depending on the scaling laws of physical parameters: the relaxation regime (with formation of microstructure) and the actuation regime (with frozen optic tensor).}
	\label{2008062138}
\end{figure}	
The elastic film can deform both in-plane through membrane deformations and out-of-plane, by bending.

\paragraph{Order tensors.}

{
According to   classical theories of liquid crystals, the description of optical axes and of order states for a cluster of nematic molecules is encoded in the eigenvalues and eigenvectors of a tensor field $Q$.
}
We define the set of biaxial (De Gennes) tensors \cite{de-gennes1993the-physics} as
\begin{eqnarray}\label{defiQQ}
\mathcal{Q}_{B}:=\Bigl\{ Q\in\R^{3\times 3}, \tr Q=0,Q=Q^T:\,\,
-\frac{1}{3}\leq\lambda_{\text{min}}(Q)\leq \lambda_{\text{max}}(Q)\leq
\frac{2}{3}
\Bigr\},
\end{eqnarray}
where $\lambda_{\text{min}}(Q)$ and $\lambda_{\text{max}}(Q)$ denote the
smallest and largest eigenvalue of the matrix $Q$. We remind that
$\mathcal{Q}_{B}$ is convex, closed and bounded. 
We introduce $\mathcal{Q}_{U}$, that is, the subset of $\mathcal{Q}_{B}$ populated by all uniaxial tensors \cite{ericksen1991liquid},
\begin{eqnarray}\label{2012131903}
\mathcal{Q}_{U}:= \Bigl\{Q\in\mathcal{Q}_{B}:\,\,
|Q|^6=54(\det Q)^2 \Bigr\}.
\end{eqnarray}

Also, we
introduce the set of (uniaxial) Frank tensors \cite{frank1958on-the-theory} which uses only the eigenframe of $Q$ as the nematic state variable, 
constrained to have eigenvalues $2/3,-1/3,-1/3$. Uniaxial tensors range
in the set
\begin{eqnarray}\label{phd012}
\mathcal{Q}_{Fr}:= \Bigl\{Q\in\mathcal{Q}_{U}:
\lambda_{\text{max}}(Q)=\frac{2}{3},\,\,
\lambda_{\text{min}}(Q)=-\frac{1}{3} \Bigr\}.
\end{eqnarray}
{We remark that \eqref{2012131903} and \eqref{phd012} are   pointwise closed and   closed in all strong topologies.}
Observe that any tensor in \eqref{phd012} can be represented in the following manner:
\begin{eqnarray}\label{2012131925}
Q= n\otimes n-\frac{1}{3} I
\end{eqnarray}
for some $|n|=1$.
It is important to remark that, whenever a liquid crystal system is described by a tensor in the form \eqref{2012131925}, then $n$
represents the common direction of the perfectly aligned nematic molecules. Instead, $\mathcal{Q}_U$ and $\mathcal{Q}_B$ describe disordered states, that is, configurations where the liquid crystal fails to be perfectly aligned. Instead, the description of such systems should be performed in probabilistic terms, and $\mathcal{Q}_U$ and $\mathcal{Q}_B$ model   probabilistic information derived from the theories of Ericksen~\cite{ericksen1991liquid} and de Gennes~\cite{de-gennes1993the-physics}, respectively.
Notice that, since $\tr Q=0$, this suffices to describe the
spectrum of $Q$. It follows by the definition that
$\mathcal{Q}_{Fr}$ is a closed and non-convex set
and the inclusion $\mathcal{Q}_{Fr}\subset\mathcal{Q}_{U}\subset\mathcal{Q}_{B}$ holds.
Importantly, $\mathcal{Q} _B$ coincides with the convex envelope of
$\mathcal{Q}_{Fr}$ and of $\mathcal{Q}_{U}$.

\paragraph{Mechanical model.}
The total energy of the system is modelled on the classical theory of linearised elasticity. Thus, we may assume  physical forces are additive and their effects are algebraically superposed.
The total energy combines a film contribution (measured on $\O_f^{\e}$) to the contribution of the nematic bonding layer (defined on $\O_b^{\e}$).
The latter, in turn, is the sum of three terms: a bulk energy density which measures the strain-order interaction of nematic elastomers according to the well-known model defined in \cite{cesana2009strain-order} and analysed in \cite{cesana2010relaxation, cesana2011nematic, cesana2011quasiconvex}; a curvature term (or Frank energy) proportional to the square of the gradient of the $Q$-tensor which, heuristically, induces molecules {to be parallel to each other;} and finally, a loading term representing the external work, the only possibly non-positive contribution to the energy.

Considering here only electrostatic work and summing all contributions, the total energy reads
\begin{multline}\label{2007291439}
	E_\e(v, Q):=
	\frac{1}{2}\int_\Oef \frac{\mathsf E^\e}{1+\nu}\left( \mmu|e(v)|^2 
	+ \lamm \tr^2	e(v) \right)dy \\
	+ \frac{1}{2}\int_\Oeb \frac{\mathsf E^\e}{1+\nu} \left(  \mmu|
	e(v)-Q|^2 + \lamm \tr^2 e(v) \right)dy  \\
	+ \frac{1}{2}\int_\Oeb K^\e_{Fr} |\nabla_\e Q|^2 dy -\frac{1}{2}\int_{\Omega^{\e}_b}\nabla\tilde{\varphi}^T{\mathsf {\tilde D}}(Q)\nabla\tilde{\varphi} dy,
\end{multline}
where admissible spaces for displacements $v$, the optic tensor $Q$, and the electrostatic potential $\tilde \varphi$ read
$$
 v\in \mathcal V_\e :=\{H^1(\O^{\e},\R^3), v( x', -\e^{p+1})=0 \text{ a.e. } x'\in \o\}, \quad Q\in  H^1(\Oeb,\mathcal{Q}_{X}),\quad \tilde{\varphi}\in H^1(\Oeb).
$$
{
Here and in what follows we adopt the notation $\mathcal{Q}_X$ 
(where $X$ stands for either $Fr,U$ or $B$) to indicate {the three} available order tensor models.
Observe that the choice of the admissible order tensor set is indeed a modelling assumption {in that, e.g.,} by constraining $Q$ to be of Frank type, we rule out biaxial order states and optical isotropy as finite-energy minimisers.}

\paragraph{Material regime (assumptions on the scaling of material parameters).} \label{par:material_regime_}
We make explicit, for definiteness, the assumptions on material parameters by fixing a parametric scaling law defining the relative elastic and nematic stiffness. 
Considering that the nematic bonding layer is much softer than the overlying film, we assume the following 
\begin{equation}
		(\mathsf E^\e, \nu^\e)(x)=\begin{cases}
	(\mathsf E, \nu), &x\text{ in }\Of\\
	(\e^q\mathsf E, \nu), &x\text{ in }\Ob 
\end{cases},\quad \text{with } q>0, 
\qquad
K^\e_{Fr} = \frac{\e^q\mathsf E}{1+\nu}\tilde\delta_{\e}^2.
\label{eqn:scalingmaterial}
\end{equation}
Here, $\mathsf E$ is the Young modulus of the elastic film and $-1<\nu<1/2$ its Poisson ratio. 
From now on, to simplify the notation without any loss of generality we assume {$\mathsf E/(1+\nu) = 1$}, leaving explicit reference to the only meaningful elastic nondimensional parameter, the Poisson ratio $\nu$. Note that this is always licit and amounts to a rescaling of displacements.
In the expression above, $\tilde\delta_{\e}$ represents the characteristic length scale which emerges from the competition between the shear modulus of nematic rubber vs. the Frank constant of the liquid crystal. For the purpose of our analysis, $\tilde\delta_\e$ identifies a critical material parameter which, as $\e$ goes to zero, may vanish or blow up, leading to the two separate regimes  of relaxation or of director actuation, respectively.
In order to bootstrap the asymptotic procedure focussing on the interplay between membrane and bending modes, we further scale dependent and independent variables as follows 
\begin{equation}\label{eqn:scalings}
v(y', y_3)=\left\{\begin{aligned}
	&L (\e \ua(L x', L\e x_3), \ut(Lx', L\e x_3))  & &\text{in } \Of\\
	&L(\e  \ua(L x', L\e^{p+1} x_3), \e^r \ut(L x', L\e^{p+1} x_3))  & &\text{in } \Ob,
\end{aligned}\right.
\end{equation}
where $r$ is the magnitude of vertical displacements, a parameter that ultimately depends on the loads.
The scaling above has a twofold goal, that of mapping the physical, $\e$-dependent domain onto a fixed, unit, domain, and that of exposing the interplay between in-plane vs. out-of-plane displacements which, in turn, depends on the type and intensity of the loads.

Similarly, we introduce the nondimensional (rescaled) electrostatic potential $\varphi$
\begin{equation}\label{2009150122}
\tilde{\varphi}(y',y_3)=\varphi_0^\e  \varphi(Lx',L\e^{p+1}x_3)\quad \textrm{in } \Ob,
\end{equation}
where $\varphi_0^\e$ is the electrostatic scale gauge.
Note that, because the electric field is solved independently from the opto-elastic problem, its scale is imparted by its boundary conditions which, in turn, can be freely chosen in such a way that the electric energy is of the same order of magnitude as the elastic terms.

\begin{figure}[t!]
	\centering
	\includegraphics[width=\textwidth]{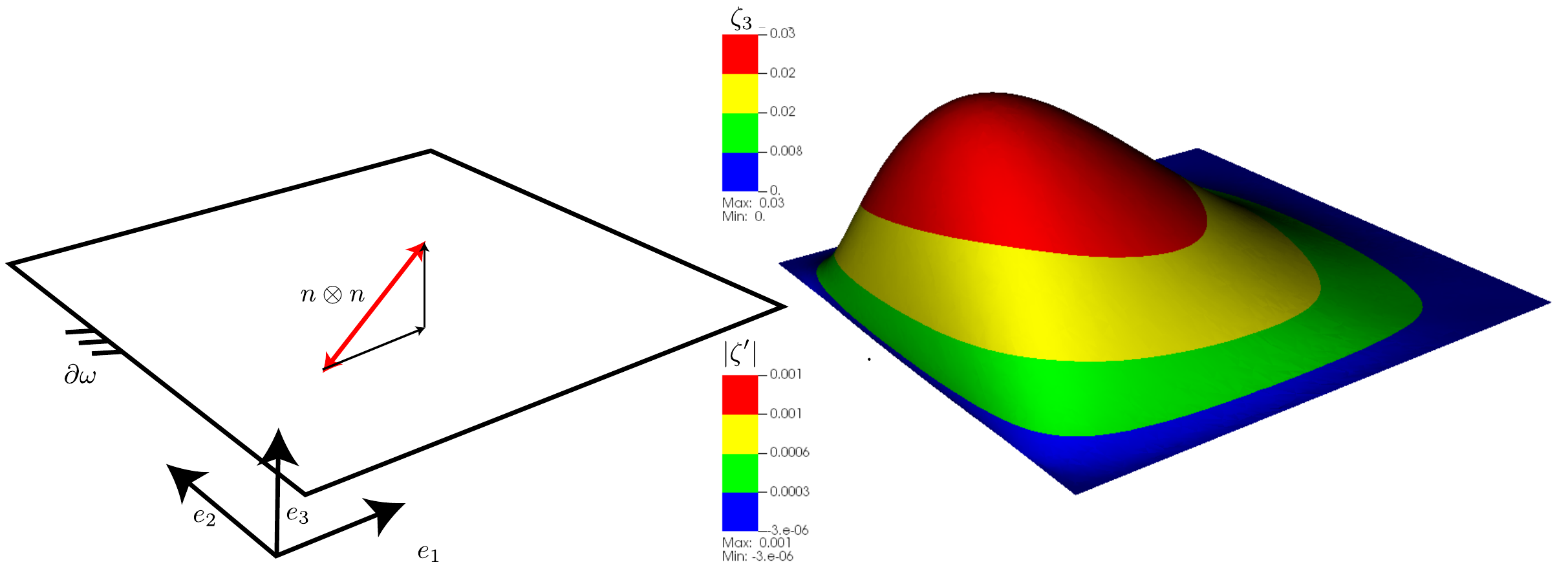}
	\caption{
	Illustrative numerical calculation of a thin nematic bilayer plate in the actuation regime, cf. Section~\ref{sec:actuation}, Theorem~\ref{lem:existenceactuation}.
The effective model given by the asymptotic theory is a fully coupled macroscopic opto-elastic plate.
	By exploiting strain-order coupling, spontaneous deformations of a multilayer composite achieve out-of-plane bending deformations under external electric stimulation.
Here, an initially flat thin active bilayer, clamped at the boundaries, is actuated by a uniform optic tensor $Q=n_0\otimes n_0-\frac{1}{3}I$ 
described by the homogeneous director
$n_0 = (\iota_1 + \iota_3)/\sqrt{2}$ (cf. image left). Colour coding in figure refers to the Euclidean norm of in-plane deformations $|\zeta'|$, the legend for $\zeta_3$ is displayed to indicate the relative scaling of transverse displacements.
	}
	\label{fig:actuation}
\end{figure}

\begin{table}[]
\centering
\begin{tabular}{@{}ll@{}}
\toprule
Symbol & Quantity \\ \midrule
$\O^{\e},\O$ & Bilayer (union of film and nematic layer)\\ 
$\O^{\e}_f, \Oeb; \Of, \Ob$  &Film and nematic layers (3D; rescaled)\\ 
$\o, L$ & Membrane planar section, diameter \\ 
$\nu, (\lambda, \mu)$ &Elastic constants: Poisson ratio, Lamé coefficients\\
$K_{Fr}^\e$  &Frank constant\\
$\epsilon_o; \epsilon_{\perp}, \epsilon_{||} $  & Dielectric constant (in vacuum);  relative constants (perpendicular and parallel)  \\ 
$ \mathsf K(\zeta',\zeta_3) $  & Effective stiffness of nematic foundation (relaxation regime)\\
$\mathsf{\tilde D}(Q), \mathsf D^{}(Q),  \mathsf {\overline D}(\overline Q)$  &Dimensional, nondimensional, and averaged matrix of dielectric coefficients \\ 
${\mathsf{ \overline B}}, \mathsf{ \overline B}(\overline Q)$  &Relaxed matrices of dielectric coefficients\\ 
${J^\e_f, J^\e_b,J^\e_{ele}}$ & Rescaled energies of film and bonding layers; electrostatic work  \\
$\mathcal J^p_\e,  J^p_\e, $ & Energy functional, mechanical model (relaxation regime)\\ 
$I_{\e}, I_0$ & Electrostatic work, asymptotic limit \\
{$\Fpe; \Fz,\Fm$} & {Energy functionals under Gauss Law, 3D and 2D limits} \\
{$\Jg, \Jmg$} & {Integral energies   under Gauss Law, 2D limits} \\

\bottomrule
\end{tabular}
\caption{Material and geometric parameters.} \label{2008061744}

\end{table}

\paragraph{Film energy.} Writing the energy~\eqref{2007291439} in terms of the scaled quantities identified in~\eqref{eqn:scalings}, the film contribution reads
\begin{multline}
\e^3L^3 J_f^\e(u)
	=\e^3L^3
	\frac{1}{2}\int_\Of \left( \mmu  |\eab(u)|^2+({\e^{-2}}\ett(u))^2 + 2|\e^{-1}\eat(u)|^2
	 +  \lamm \left(\eaa(u)+ \e^{-2}\ett(u) \right)^2 
 \right)dx. 	\label{2008072345}
\end{multline}

\paragraph{Nematic energy.} 
On the other hand, using~\eqref{eqn:scalings}, the nematic contribution to the total energy~\eqref{2007291439} reads
\begin{multline}\label{eqn:nemenergy}
\e^3 L^3\e^{q+p-2}J^\e_b(u,Q):= \\ 
L^3\frac{1}{2}\e^q\e^{p+1}\int_\Ob \mmu \left( |\e e'(u)-Q'|^2+ \left( \frac{\e^{r}}{\e^{p+1}}\ett(u)-\Qtt \right) ^2 + 2\left| \frac{1}{2} \left( \e^r \nabla' \ut + \frac{\e}{\e^{p+1}}\pt u' \right) -(Q \iota_3)' \right|^2\right) dx\\ 
+  L^3\frac{1}{2}\e^q\e^{p+1}\int_\Ob\lamm \left( \e\tr(e'(u))+ \frac{\e^r}{\e^{p+1}}\ett(u) \right)^2dx+
\frac{1}{2}L\int_\Ob K^\e_{Fr} \left( |\nabla' Q|^2+\left| \frac{\pt Q}{\e^{p+1}} \right|^2  \right)\e^{p+1} dx,
\end{multline}
where the last term, the curvature energy, rewrites
\begin{multline}
\frac{1}{2}L\int_\Ob K^\e_{Fr} \left( |\nabla' Q|^2+\left| \frac{\pt Q}{\e^{p+1}} \right|^2  \right)\e^{p+1} dx=
\frac{1}{2}L\e^{p+1+q}\tilde\delta_\e^2\int_\Ob  \left( |\nabla' Q|^2+\left| \frac{\pt Q}{\e^{p+1}} \right|^2  \right) dx=\\
\frac{1}{2}L\frac{\e^{p+1+q}\tilde \delta_{\e}^2}{\e^{2p+2}}\int_\Ob  \left(\e^{2p+2}|\nabla' Q|^2+\left| \pt Q \right|^2  \right) dx.
\end{multline}{In the expression above and throughout the paper, $(Q \iota_3)'$  is a planar vector whose components are $\Qat$.}
Here, with no loss of generality, we have adopted the so-called single-constant approximation of the full nematic curvature energy, collapsing to one single deformation term the splay, bend, and twist modes (see \cite{virga1995variational}, \cite{de-gennes1993the-physics}).

We choose to keep an explicit dependence upon $p$ because, depending on its value (the relative thickness of the film layers), we identify phenomenologically different limit regimes.
{Note that the curvature energy penalises spatial reorientation of the order tensor and in turn, pattern formation. Owing to linearity this contribution is integrated over the reference domain $\Ob$. 
We refer to~\cite{BDS15} for an approach based on non-linear Frank model whereby the curvature penalisation is measured on the deformed configuration.}

\paragraph{Electrostatic work.}
The electrostatic work density is the (scalar) product of the electrostatic vector $\mathbb{E}$ by the dielectric displacement vector $\mathbb{D}$.
Based on the linear model of nematic liquid crystals, the relation between $\mathbb{D}$ and $\mathbb{E}$ is obtained introducing the tensor of dielectric coefficients $\mathsf{\tilde D}(Q)$ so that $\mathbb{D}:=\mathsf{\tilde D}(Q)\mathbb{E}$, see~\cite{de-gennes1993the-physics}, \cite{longa1987an-extension}.

Upon introduction of the  
electrostatic potential, related to the 
electric field by
$\mathbb{E}=-\nabla \tilde\phi$, the electrostatic work density is given by ${\frac{1}{2}}\mathbb{D}\cdot \mathbb{E}={\frac{1}{2}} \mathbb{E}^T\mathsf{\tilde D}(Q)\cdot\mathbb{E}$ where $ \tilde{\mathsf{{D}}}(Q)=\epsilon_0\mathsf D(Q)$.
The tensor of dielectric coefficients depends linearly on the order tensor $Q$. In turn, the electric field is obtained by optimisation and depends, in an intricate way, upon $Q$. We shall elaborate on their connection in the Actuation, Section~\ref{sec:actuation}.
The scaled electrostatic work reads

\begin{equation}\label{2008072347}
{\frac{1}{2}}\int_{\Omega^{\e}_b}\nabla\tilde{\varphi}^T\mathsf {\tilde D}(Q)\nabla\tilde{\varphi}\,dx=
\e^3	(\phi_0^\e)^2\e^{p-2}\epsilon_0 L{\frac{1}{2}}\int_{\Ob} \nabla_\e^T  \varphi  \mathsf D(Q)\nabla_\e  \varphi
\, dy=:
\e^3 (\phi_0^\e)^2\e^{p-2}\epsilon_0 L J^\e_{ele}(Q, \phi).
\end{equation}
where we have concisely denoted  by $\nabla_\e \varphi$ the scaled gradient of a scalar function, namely $\nabla_\e \varphi:=\left(\nabla'\varphi,\frac{1}{\e^{p+1}}\partial_3\varphi\right)$.
Note that the work integral above, for a fixed $Q\in \mathcal{Q}_X$ ($X$ being a place-holder for $Fr, U$ or $B$) and for a fixed $\e>0$, is a standard elliptic functional modelled upon the symmetric positive definite matrix of nondimensional dielectric coefficients $\mathsf D(Q)$.

\paragraph{Additional mechanical loads.}
Finally, we consider applied body and surface loads by prescribing two force densities, $f^\e$ in the interior and $g^\e$ on the upper face of the film domain.
Without loss of generality, we scale imposed loads in such a way that the corresponding work is of the order of magnitude of the elastic film energy.
Accordingly, we set $f^\e := \e^2 f, g^\e := \e^3 g $ with $f:L^2(\Of, \R^3), g:L^2(\o\times \{1\}, \R^3)$ so that the scaled linear form corresponding to mechanical loads (cf., e.g., s~\cite[vol.~2]{ciarlet1988three-dimensional}) reads
\begin{equation}
\int_{\Oef} f^\e v dy+\int_{ \o\times \{\e\}} g^\e v dy =\e^3 L^3 \int_{\Of} f u dx+ \e^3L^2 \int_{\o\times \{1\}} gu dx'.
\end{equation} 
{Here we have used~\eqref{eqn:scalings}.}
Note that such assumption on the scaling of loads is not restrictive owing to the fact that the mechanical work is a continuous perturbation to the total energy.

\subsection{Scaling regimes} \label{par:scaling}
We specialise the scaling laws introduced in~\eqref{eqn:scalingmaterial},~\eqref{eqn:scalings} in order to focus on the material regime in which there is possible coupling between membrane and bending deformation modes, as well as with the optoelastic behaviour of the nematic layer.
Heuristically, the bending energy of the film scales like $\e^3$, thus we fix the scaling parameters of the system in such a way that both i) the energy of the nematic bonding layer is of the same order of magnitude of the bending energy of the film, ii) we focus on vertical displacements which are of the same order of the thickness of the overlying film, and iii) the electrostatic work is of the same order of magnitude of the membrane energy of the film. Respectively, we set
\begin{equation}\label{2009011817}
\text{i) }	q+p-2=0, \qquad\text{ii) } r=p+1,\qquad \text{iii) }\phi_0^\e=L\e^{-1-p/2} {\epsilon_0}^{-1/2}.
\end{equation}
Under these assumptions, the total energy, i.e., the sum of film and nematic layer energies minus the external work, as defined in 
\eqref{2008072345},
\eqref{eqn:nemenergy}
and 
\eqref{2008072347},
reads
\begin{multline}
\label{2009051223}
\Jepen(v,Q,\varphi) : = 
J_f^\e(v)+J^\e_b(v,Q)-J^\e_{ele}(Q,\varphi)
\\
	=\frac{1}{2}\int_\Of \left( |e'(v)|^2+({\e^{-2}}\ett(v))^2 + 2|\e^{-1}(e(v)e_3)'|^2\right) dx + \lamm \left((\tr(e'(v)))+ \e^{-2}\ett(v) \right)^2dx\\
	+\frac{1}{2} \int_\Ob  \left( |\e e'(v)-Q'|^2+ 
	\left( \ett(v)-\Qtt \right) ^2 \right)dx \\
	+ \frac{1}{2}\int_\Ob 2\left| \frac{1}{2} \left( \e^{p+1} \nabla' \vt + \e^{-p}\pt v' \right) -(Q \iota_3)' \right|^2dx+
	\lamm \left( \e\tr(e'(v))+ \ett(v) \right)^2dx\\
	+ \frac{1}{2}\int_\Ob \frac{\tilde\delta_\e^2}{L^2\e^{2p+2}} \left(\e^{2p+2}|\nabla' Q|^2+\left| {\pt Q} \right|^2  \right)dx
	-
{\frac{1}{2}}\int_{\Ob}\left(\nabla'\varphi,\frac{1}{\e^{p+1}}\partial_3\varphi\right)^T\mathsf D(Q) \left(\nabla'\varphi,\frac{1}{\e^{p+1}}\partial_3\varphi\right) dx.
\end{multline} 
The quantity 
\begin{equation}
	\delta_\e^2 :=\frac{ \tilde\delta_\e^2}{L^2\e^{2p+2}},
\end{equation}
identifies a material length scale stemming from the ratio between Frank's curvature constant and the bonding layer's stiffness, relatively to the size of the domain $L$ and the thickness of the nematic layer.
Notice that this quantity is scaled with respect to the thickness, hence, depending on the material regime and geometric dimensions may either vanish or blow up, as $\e\to 0$.
These two scenarios indeed correspond to the  distinct material regimes of actuation (with fixed orientation of the director) and that of spontaneous relaxation (with emergence micro-textured patterns).

More precisely,
the relaxation scenario is dominated by the \emph{rescaled} length scale $\delta_\e$, in the regime  $\delta_\e\to 0$. In this setting, $\delta_\e$ is the smallest scale of the system well below the layers' thickness and allows for  transition layers of negligible energetic cost.
Contrarily, the actuation regime is characterised by the \emph{macroscopic} length scale $\tilde\delta_\e$, in the limit $\tilde\delta_\e\to \infty$. In this context the optic tensor is  rigid, its homogeneity is forced under the influence of applied external fields.

Because an electric field generated by an external device acts on the nematic elastomer by orienting the LC molecules and thus performing work, the sign of the functional is undefined. A careful analysis is required to study critical points of the total energy which are of saddle-type.
We devote Section~\ref{sec:actuation} to the analysis of the nematic elastic foundations and electric fields, 
whereas we focus our attention in the next section to the analysis of the regime of nematic relaxation where optoelastic patterns spontaneously emerge, without external stimuli, in such a way to relax mechanical stresses. 
Accordingly, we set $\varphi\equiv 0$ in~\eqref{2009051223}
and compute the asymptotics as $\e\to 0$ of the energy ${J_\e^p(v,Q,0)}$.

\section{Relaxation}\label{sec:relaxation}
The relaxation regime for nematic multilayers is characterised by the spontaneous emergence of textured microstructures and a strong two-way coupling between optic axis and elastic displacements.
This scenario, in turn, occurs as Frank's curvature energy is small and transitions between differently oriented microscale domains can be accommodated with little energetic cost. Indeed, in this case, Frank's stiffness provides the smallest length scale of the system.
In line with the modelling approach introduced for micromagnetics~\cite{de-simone1993energy}, 
relaxation occurs as $\delta_{\e}$ vanishes, corresponding to the regime of a large plates with a small bending constant.

The program is to explicitly compute the effective stress relaxation induced by a local accommodation of the optical texture under mechanical deformation, a mechanism which is responsible of the emergence of fine scale, possibly periodic, optical patterns {of martensitic type (see \cite{DeS99}, \cite{Bhatta03})}. 
In energetic terms, this amounts to first computing locally-optimal nematic textures at microscale and then performing the dimension reduction to derive the following macroscopic two-dimensional one-variable model  
\cite{cesana2018variational,cesana2015effective,cesana2010relaxation,desimone2002macroscopic}

\begin{equation}
	\mathcal{J}^p_\e(u):=
	\begin{cases}
	\displaystyle{\inf_{Q\in H^1(\O_b, \mathcal{Q}_{X})} {J}^p_\e(u, Q, 0)} &\text{if }u\in \mathcal V \\
		+\infty &\text{if } u\in L^2(\Omega,\R^3)\setminus \mathcal V,
	\end{cases}
\label{eqn:total3denergy}
\end{equation}
where 
\begin{equation}
\label{eqn:3dspace}
	\mathcal V := \{ H^1(\O,\R^3), u(x',-1)=0\}
\end{equation}
 is the set of  kinematically admissible three-dimensional displacements 
and 
$X$ stands for either $Fr,U$ or $B$, depending on the underlying  order tensor model.
By computing and matching a lower and an upper bound, 
we show that, the $\Gamma$-limit of \eqref{eqn:total3denergy} is the same to all {nematic order} models and corresponds to an effective biaxial order-tensor model.

{Our work builds upon  \cite{cesana2018variational} where the  strain-order coupling 
has been fully explored and clarified for effective models of NLCE  bilayers in planar confinement.}
{
Here, we focus on the specific aspects related to the coupling between in-plane and out-of-plane displacements,
pertaining to the compactness and the characterisation of the limit space for energy minimising displacements, as well as its mechanical role.
We elaborate and we give full account of this in our proof of the Gamma-liminf inequality.
The self-contained proof of the Gamma-limsup inequality is postponed to the Appendix, adapting the result in~\cite{cesana2018variational} to the present situation.
}

\begin{Remark}
{We may rewrite~\eqref{eqn:total3denergy} in compact notation} by introducing scaled strain tensors in the film $\kappa_\e$ and in the nematic layer $\hat \kappa_\e$, reading respectively
\begin{equation}\label{2007092341}
\kappa_\e(u)=
\left( 	\begin{matrix}
		e'(u) & {\frac{1}{2\e}}\left( \nabla' \ut + \pt u' \right)\\
		\text{sym} & \e^{-2} \ett(u) 
	\end{matrix} \right) \quad \text{and} \quad
\hat\kappa_\e(u)=
\left( 	\begin{matrix}
		e'(u) & {\frac{1}{2\e}}\left( \e^{p+1}\nabla' \ut + \e^{-p}\pt u' \right)\\
		\text{sym} & \ett(u) 
	\end{matrix} \right),
\end{equation}
 For $u\in \mathcal V$
\begin{multline}
\label{2007291555}
	J^p_\e(u)=
		\frac{1}{2} 
		\int_\Of \left( |\kappa_\e|^2 + \lamm \operatorname{tr}^2 \kappa_\e \right)dx  \\
		+ \frac{1}{2}\inf_{Q\in H^1(\O_b, \mathcal{Q}_{X})}
		\int_\Ob \left( |\hat \kappa_\e - Q|^2 +\lamm \operatorname{tr}^2 \hat \kappa_\e+ \frac{1}{2}  \delta_\e^2 \left( \e^{p+2}|\nabla' Q|^2+\left| {\pt Q} \right|^2  \right) \right)dx .
\end{multline}
\end{Remark}
The convergence properties of minimising sequences of displacements associated to the functional above characterise the limit space of displacements, independent of the thickness ratio $p$.
However, it is \emph{the rate} of convergence of minimising sequences (depending on values of $p$)  that identifies the contributions entering into the limit asymptotic models. For this reason, we  carry explicit dependence on $p$ in the total energy functional.
Also note that, because it is the boundedness of scaled terms that implies sharp convergence properties of displacements, the formulation via rescaled strains~\eqref{2007291555} proves to be effective in clarifying and rendering explicit the compactness of minimising sequences in \eqref{eqn:total3denergy}.

\subsection{Estimates and compactness}\label{par:estimates}
We start with two preliminary results frequently invoked in the reminder of the article.

\begin{lemma} [Poincaré-type inequality]
\label{lem:poincare}
Let $u\in L^2(\O,\R^3)$, with $\O=\o\times (-1,1)$, 
{$\partial_3 u\in L^2(\O)$ with $u(x', -1)=0$} a.e. $x'\in \o$. Then there exists a constant $C>0$ depending only on $\o$, such that
$$
\nltwo[u]{L^2(\O)}^2\leq C \nltwo[\pt u]{L^2(\O)}^2
$$
\end{lemma}

The next result, proved in~\cite[Section 4.1]{cesana2018variational}, allows to characterise the weak limit of the (suitably rescaled) gradient of a bounded displacement field within the nematic layer.

\begin{lemma}[Convergence of gradients]
Let $f_{\e}\in H^1(\Ob,\R^3)$ for every $\e$. Let $K>0$.
Suppose $f_{\e}\in L^2(\Ob,\R^3)$ uniformly bounded in $\e$ and
$
\e^K\|\nabla 'f_{\e}\|_{L^2(\Ob,\R^3)}	  \leq C,
$
with $C$ independent of $\e$. Then
$
\e^K \partial_i f_{\e} \wto 0
$
weakly in $L^2(\Ob,\R^3)$, for $i=1,2,3$.
\label{lemma:2008081840}
\end{lemma}

 \begin{proof}
 See Paragraph Compactness of Section 4.1 in \cite{cesana2018variational}.
 \end{proof}

Considering admissible minimising sequences $(u^\e) \subset L^2(\O,\R^3)$ that leave the energy uniformly bounded
implies,
thanks to  Lemma~\ref{lem:poincare}, 
 the uniform boundedness of three-dimensional displacements in $L^2(\O,\R^3)$.
Therefore, there exists a compact set of $L^2(\O,\R^3)$ such that minimising sequences are compact therein.
The two lemmas above allow to establish the following characterisation of limit strains.
{In what follows, we denote thickness averages by an overline (cf. notation in Section~\ref{par:notation}). Also, observe thanks to Jensen's inequality we have, for $f\in L^2(\Of)$, that
$
\|f\|_{L^2(\Of)}\ge \|\overline{f}\|_{L^2(\o)}.
$

\begin{proposition}[Characterisation of limit strains]\label{2009141558}
Consider a sequence $\ue\subset L^2(\Omega,\R^3)$ for every $\e$
such that $\ue(\cdot, -1)=0$
and $\ue\to u$ strongly in $L^2(\Of,\R^3)$ as $\e\to 0$ and plug $\ue$ into ${J_\e^p(\ue)}$. 
Uniform boundedness ${J_\e^p(\ue)\leq C}$ implies that

\begin{itemize}
 \item[a)] there exists a limit $\hat k\in L^2(\Ob, \R^{3\times 3})$ such that   $\hat \kappa_\e\wto \hat k$ in  $L^2(\Ob, \R^{3\times 3})$.
\item[b)]
 there exists $ k\in L^2(\Of, \R^{3\times 3})$ such that $\kappa_\e^{} \wto k$ in $L^2(\Of, \R^{3\times 3})$, and
$$
 k_{33} = -\frac{\nu}{1-\nu }e_{\alpha \alpha}(u) , \qquad
k_{\alpha \beta}=\eab(u),	
 $$
		\item[c)] $\eit(\ue)\to 0$ strongly in $L^2(\Of,\R^{3\times 3})$,
\item[d)] there exists $e\in L^2(\Of,\R^{3\times 3})$ such that
			$e'(\ue)\wto e'$ weakly in $L^2(\Of,\R^{3\times 3})$,
		\item[e)] $\nltwo[\ett(\ue)]{L^2(\Ob,\R^{3\times 3})}\leq C$,
		\item[f)] $e'(\ue)\wto 0$ weakly in $L^2(\Ob,\R^{2\times 2})$,
\item[g)] $\e^{p+1}\nabla ' \overline{u}_3^{\e}\rightharpoonup 0$, weakly in $L^2(\o)$.
	\end{itemize}
\end{proposition}

\begin{proof} To carry the proof of the items above we systematically use Jensen's inequality to obtain lower bounds upon integration of a convex function along the thickness, as in $||\bar f||_{L^2(\o)} \leq ||f||_{L^2(\Ob)},$ for all$ f\in L^2(\Ob)$, where the overbar stands for the thickness average.
Item $a)$ simply follows from the uniform boundedness of $\|\hat{ \kappa}^{\e}\|_{L^2(\Ob,\R^{3\times 3})}$, and the boundedness of $Q$.
{Furthermore, $b)$ derives from the uniform boundedness of $\| \kappa^{\e}\|_{L^2(\Of,\R^{3\times 3})}$ by optimising with respect to the $k_{33}$ component and noticing that the convergence of minimising sequences $\ue$ in $\Of$ is actually weak $H^1(\Of, \R^3)$, hence limit rescaled strains can be identified with scaled components of the limit strain.}
To prove $c)$ observe that $b)$ implies the existence of constants $C, C'$ such that
\[
	\|\e^{-1}e_{\alpha 3}\|_{L^2(\Of,\R^{3\times 3})}\le C, \text{ and }
		\|\e^{-2}e_{33}\|_{L^2(\Of,\R^{3\times 3})}\le C'.
\]
To prove $d)$ observe that {$b)$} implies
		$\|e'(\ue)\|_{L^2(\Of,\R^{3\times 3})}\le C$.
Then, $e)$ is implied by {$a)$}.
To prove $f)$ we need to use
 $\|e'(\ue)\|_{L^2(\Ob,\R^{3\times 3})}\le C$, uniformly in $\e$ (implied by $a)$), and then invoke Lemma~\ref{lemma:2008081840}.

To prove $g)$
we first claim the following: there exist constants $C, C'$ such that
\begin{equation}\label{2007231651}
	\e^{p+1} \nltwo[\nabla' \avguet]{L^2(\Ob)}\leq C \text{ and }\e^{-p} \nltwo[\pt\avguea]{L^2(\Ob)}\leq C'.
\end{equation}
These terms vanish in the limit energy owing to the boundedness of the gradient terms and the fact that they are multiplied by a vanishing sequence.
To establish the estimate \eqref{2007231651} it suffices to integrate the energy estimate for the shear term exploiting convexity and use $a)$.
Indeed, 
\begin{equation}
	\nltwo[\e^{p+1} \nabla' \avguet + \e^{-p}\pt\avguea]{L^2(\o)}\leq\nltwo[\e^{p+1}\nabla' \uet + \e^{-p}\pt {\uea}]{L^2(\Ob)}\leq C,
\end{equation}
then use triangle inequality.
By explicit integration we obtain a boundary norm whose boundedness in $H^1(\o)$ is ensured by the compactness of trace operator, the continuity of displacements, and their weak convergence through the use of the trace theorem \cite[Theorem 6.1-7]{ciarlet1988three-dimensional}.
We can thus write
\begin{equation}\label{2007232328}
\e^{p+1} 	\nltwo[\nabla' \avguet ]{L^2(\o)}\leq \nltwo[\e^{p+1} \nabla' \avguet + \e^{-p}\pt\avguea]{L^2(\o)}+ \nltwo[\e^{-p}{\uea}(x', 0)]{L^2(\o)}\leq C,
\end{equation}
where we have used  $\pt\avguea = \ue(x', 0)$ by virtue of boundary conditions. Hence $\nabla' \avguet $ goes to zero weakly in $L^2(\omega)$ thanks to Lemma~\ref{lemma:2008081840}.
In \eqref{2007232328}, notice that $ {\avguea}(x', 0)\to \ua(x', 0)$ strongly in $L^2(\o)$ by the trace theorem and therefore
$\nltwo[ {\uea}(x', 0)]{L^2(\o)}$
is uniformly bounded in $\e$.
\end{proof}

\subsubsection{Kirchhoff-Love sets of displacements
$KL$ and $KL^{\sharp}$}\label{2009141551}

The structure of limit displacements is determined upon integration with respect to $z_3$ of the film relations (see $c)$ in Prop. \ref{2009141558})
\begin{equation}\label{2006231544}
\ett=0\Longrightarrow \partial_3 u_3=0, {\qquad \eat(u) =0\Longrightarrow \partial_\alpha u_3=- {\partial_3}{ u_\alpha}.}
\end{equation}
The first implies that $u_3$ is a function of $x'$ only, that is $u_3(x)=\zeta_3(x')$. For such functions, the latter relations yield,
upon integration in $x_3$,
\begin{equation}\label{2012161650}
(u'(x',x_3),u_3(x'))= (\zeta'(x')-x_3\nabla'\zeta_3(x'),\zeta_3(x')).
\end{equation}
These relations identify the limit space as the set of (Kirchhoff-Love) displacements
\begin{eqnarray}\label{eqn:kl}
KL:=\{v\in H^1(\Of, \R^3) : v'= \zeta'-x_3 \nabla' \zeta_3, v_3=\zeta_3,
\text{ with }  \zeta'\in H^1(\omega, \R^2),   \zeta_3\in H^2(\omega) , x_3\in (0,1) \}
\end{eqnarray}
which is equivalent (cf.~\cite{ciarlet1988three-dimensional}) to set of functions $u\in \mathcal V$ for which \eqref{2006231544} holds.
Observe then that items $c)$, $d)$ of Proposition~\ref{2009141558}), and Korn's inequality
\cite[Theorem 6.3-3]{ciarlet1988three-dimensional} 
imply the weak convergence of $u_\e$ to a certain $u^\star\in KL$.
In the definition above, $\zeta'$ coincides with the trace of the three-dimensional displacement $u$ at interface between the two layers $\o \times \{0\}$.
By analogy, we introduce the set of \emph{shifted} limit displacements 
\begin{eqnarray}\label{eqn:klsharp}
KL^{\sharp}:=\{ v\in H^1(\Of, \R^3) : v'= \zeta'_\sharp-(x_3-\frac{1}{2}) \nabla' \zeta_3, v_3=\zeta_3,
\text{ with }  \zeta'_\sharp\in H^1(\omega, \R^2), \zeta_3\in H^2(\omega) , x_3\in (0,1) \},
\end{eqnarray}
where the functions $\zeta'_\sharp$ represent the trace of the three-dimensional displacement $u$ in correspondence to the mid-section of the film $\o \times \{1/2\}$.
Note that, from the topological and functional standpoint $KL$ coincides with $KL^\sharp$ and the functions representing in-plane displacements are related by
\begin{equation}
\label{eqn:KLchange}
\zeta'(x') = \zeta'_\sharp(x') +\frac{1}{2}\nabla' \zeta_3(x'), \qquad \text{a.e. } x'\in \o.
\end{equation}

\subsection{Gamma-limits of nematic plate foundations}
We now turn to the mathematical analysis and mechanical discussion of the two physically relevant material regimes, as a function of the aspect ratio represented by $p$, referred to as the  `thin nematic', for $p=0$, and the `thick nematic'  case, for $-1< p< 0$.
This first setting leads to a full opto-elastic coupling between the nematic layer and the overlying elastic plate, whilst the second scenario involves only a partial (transverse) opto-elastic coupling.
The following is the main result of this section.
\begin{theorem}[Fully coupled, thin nematic]
\label{thm:relaxcoupled}
Let
$\Jez$ be
the energy defined in
$\eqref{eqn:total3denergy}$, with $p=0$.
Then, 
$$
\Jz(u^\star)=\Gamma\hbox{-}\lim_{\e\to 0}\Jez (u^\star)
$$
in the strong $L^2(\Of, \R^3)$-topology, 
where  $u^\star  = (\zeta'(x')-x_3\nabla '\zeta_3(x'),\zeta_3(x')) \in KL$,
\begin{equation}
\Jz(u^\star)=
	\begin{cases}
	\displaystyle{ 
	\frac{1}{2}\int_\o }
	 \left( |e'(\zeta')|^2 - e'(\zeta)\nabla'\nabla' \zeta_3 + \frac{1}{3}|\nabla'\nabla' \zeta_3|^2 \right) dx' \\
	 \qquad+ 
	\displaystyle{
	\frac{1}{2}\int_\o }\coefopt\left((\tr (\zeta'))^2- \tr (\zeta') \Delta' \zeta_3 + \frac{1}{3}(\Delta' \zeta_3)^2 \right)dx' 
\\
\displaystyle{
+\frac{1}{2}\int_{\o}\left(\operatorname{dist}^2(\mathsf K(\zeta',\zeta_3), \mathcal{Q}_B)+ \lamm\zeta_3^2 \right)dx'}
&\textrm{ if }(\zeta', \zeta_3)\in  H^1(\o, \R^2)\times H^2(\o)\\ 
+\infty &\text{otherwise in} \in L^2(\omega,\R^3),
\end{cases}
	\label{eqn:relaxfullcouple}
\end{equation}
\begin{equation}\label{2012251700}
	\mathsf K(\zeta',\zeta_3) = \left(
	\begin{matrix}
	0 & 0 & \frac{1}{2}\zeta_{1}\\
	0 & 0 & \frac{1}{2}\zeta_{2}\\
	\frac{1}{2}\zeta_{1} & \frac{1}{2}\zeta_{2} & \zeta_3
	\end{matrix}\right),
\end{equation}
and
\begin{equation}
\operatorname{dist}^2( \mathsf{K}(\zeta',\zeta_3), \mathcal{Q}_B)=
 \inf_{\overline{Q}\in  \mathcal{Q}_B } 
| \overline Q- \mathsf{K}(\zeta',\zeta_3)|^2.
\end{equation}
\end{theorem}

\begin{proof}
We  compute and match a  lower bound (the $\Gamma$-liminf inequality) 
{with a suitably constructed upper bound 
($\Gamma$-limsup inequality) 
to $\Gamma\hbox{-}\lim_{\e\to 0} \Jez(u)$. }

Imposing $p=0$ in Proposition~\ref{2012031723} we have the $\Gamma$-liminf inequality. 
Then, from Proposition~\ref{prop:glimsupp0} fixing $p=0$, we obtain the $\Gamma$-limsup inequality and the result follows.
\end{proof}

\begin{theorem}[Weakly coupled, thick nematic]
\label{thm:relaxweaklycoupled}
Let $-1<p<0$ and $\Jep$ as in
\eqref{2009051223} and \eqref{eqn:total3denergy} respectively.
Then, 
$$
\Jm (u^\star)=\Gamma\hbox{-}\lim_{\e\to 0} 
\Jep
(u^\star)
$$
in the strong $L^2(\Of, \R^3)$-topology, 
where {$u^\star = (\zeta'_\sharp-(x_3-\frac{1}{2}) \nabla' \zeta_3,\zeta_3) \in KL^\sharp$
},
\begin{equation}\label{2012242240}
\Jm(u^\star)=
	\begin{cases}
	\displaystyle{
\frac{1}{2}\int_\o 
	\left( |e'(  \zeta'_\sharp)|^2+ \frac{1}{3}|\nabla'\nabla'\zeta_3|^2+\coefopt
	\left( \operatorname{tr}^2e'( \zeta'_\sharp) +\frac{1}{3}(\Delta'\zeta_3 )^2 \right) \right) dx'}\\
\qquad\displaystyle{+\frac{1}{2}\int_{\o}\left(\operatorname{dist}^2(  \mathsf{K}(0,\zeta_3), \mathcal{Q}_B)+ \lamm\zeta_3^2 \right)dx' }
	 \qquad\text{if } (\zeta_{\sharp}', \zeta_3),\in   H^1(\o, \R^2)\times H^2(\o)\\
		+\infty, \qquad \text{ otherwise in } L^2(\omega,\R^3)
	\end{cases}
\end{equation}
where  $\mathsf{K}$ is defined in 
\eqref{2012251700}.
\end{theorem}

Theorem 
\ref{thm:relaxweaklycoupled}
is a consequence of  
Proposition
\ref{2012031723}
($\Gamma$-liminf inequality)
and of
Proposition
\ref{prop:glimsupp0}
($\Gamma$-limsup inequality)
for a functional defined 
on displacements at the mid-section of the film.
The proof of Theorem
\ref{thm:relaxweaklycoupled}
is postponed to Section
\ref{2012262313}.

\subsection{Proof of Gamma-convergence theorems  for $-1<p\le 0$}

We analyse thin and thick models of nematic foundations condensing two relaxation results.
Propositions \ref{2012031723} (lower bound) and \ref{prop:ub-allp} (upper bound)
 suffice to characterise 
$\Gamma$-limits 
for nematic foundations for $-1<p\le 0$ comprehensively,
by characterising the asymptotic plate regime in terms of 
$KL$-displacements measured at the interface between nematic and film layer.
While this is precisely the requested result for thin nematic foundations ($p=0$),
we are left with performing a final shift mapping 
from $KL$ to $KL^{\sharp}$
  to represent the $\Gamma$-limit
in terms of
the  film mid-section
for plates with thick foundations without shear coupling ($-1<p<0$). This is done 
in Section \ref{2012262313}.

\begin{proposition}[Lower bound inequality]\label{2012031723}
Consider  $\Jep$ as  in $\eqref{eqn:total3denergy}$, for $-1<p\le 0.$
Then for
 sequences $(\ue)\subset L^2(\O, \R^3)$ 
converging to $u^\star $ strongly in $L^2(\Of, \R^3)$
we have
\begin{eqnarray}\label{prop:gliminfp0}
\Gamma\hbox{-}\liminf_{\e \to 0}\Jep( u^\star) \geq  \Jp( u^\star),
\end{eqnarray}
where
\begin{equation}\label{2012261730}
\Jp( u^\star)=
	\begin{cases}
\displaystyle{\frac{1}{2}\int_\o 
	 \left( |e'(\zeta')|^2 - e'(\zeta')\nabla'\nabla' \zeta_3 + \frac{1}{3}|\nabla'\nabla' \zeta_3|^2\right ) 	dx' } \\
	 \qquad+ 
\displaystyle{	\frac{1}{2}\int_\o  \coefopt \left(\operatorname{tr}^2 e'(\zeta') - \operatorname{tr} e'(\zeta') \Delta' \zeta_3 + \frac{1}{3}(\Delta' \zeta_3)^2 \right) dx'}
\\
\displaystyle{
\quad +\frac{1}{2}\int_{\o}\left(\operatorname{dist}^2(\mathsf K((\zeta^{[p]})',\zeta_3), \mathcal{Q}_B)+ \lamm\zeta_3^2 \right)dx',}
\qquad \textrm{ if }(\zeta', \zeta_3)\in  H^1(\o, \R^2)\times H^2(\o)\\ 
+\infty, \qquad\text{otherwise in} \in L^2(\omega,\R^3)
\end{cases}
\end{equation}
{where $u^\star=(\zeta'(x')-x_3\nabla'\zeta_3(x'),\zeta_3(x'))$ and} we write $(\zeta^{[p]})'=\zeta'$ if $p=0$
and
  $(\zeta^{[p]})'=0$ if $p\in (-1,0)$.
\end{proposition}

\begin{proof}

We consider a general  sequence $(\ue)\subset L^2(\O, \R^3)$ 
converging to $ u^\star$ in $L^2(\Of, \R^3)$ and
such that
$\Jep(\ue)$ is uniformly bounded in $\e$. Thanks to Proposition \ref{2009141558}, it necessarily follows
 $\ue \to u^\star\in KL$ 
 strongly in $L^2(\Of, \R^3)$ and we have
\begin{multline}\label{2007231738}
	\liminf_{\e\to 0}\Jep(\ue)
\geq \liminf_{\e\to 0}\frac{1}{2}\left \{ \int_\Of \left[|e'(u_{\e})|^2 + \coefopt  \tr^2(e'(u_\e))  \right]dx + 
	\inf_{Q\in H^1(\Ob,\mathcal{Q}_{X}) }  \int_\Ob \Bigl[
	|e'(\ue)-Q'|^2 + \right.\\	
\left.	2|\tfrac{1}{2} ( {\e^{-p}}\pt{\ue}' + \e^{{p+1}}\nabla'\ue_3  ) -\Qat|^2 + \lamm ( \e \tr(e'(\ue))+\ett(\ue)  )^2 + (\ett(\ue)-\Qtt)^2 \Bigr]dx\right\}.
\end{multline}
The inequality in~\eqref{2007231738} is obtained by neglecting  shear terms in film, 
and optimising with respect to transverse component $e_{33}$ of the strain gradient in the film layer, which implies
\begin{equation}
	\frac{1}{\e^2}e_{33}(\ue)=	\coeftransvopt \tr e'(\ue).
\end{equation}
Integrating with respect to $x_3$ applying Jensen's inequality, we expose all averaged quantities (indicated by an overhead bar). We obtain a further lower bound by extending the optical minimisation from $H^1(\Ob,\mathcal{Q}_{Fr}) $ to the larger $L^2(\Ob,\mathcal{Q}_{B})$.
Taking vertical average in $\Ob$ leads to 
\begin{multline}
\liminf_{\e\to 0}\Jep(\ue)
\geq \liminf_{\e\to 0}\left \{ \frac{1}{2}\int_\Of \left[|e'({u}_{\e})|^2 + \coefopt ( \tr(e'({u}_{\e}))^2  \right]dx + 
	\inf_{Q\in L^2(\Ob,\mathcal{Q}_{B}) } \frac{1}{2}\int_\o \Bigl[
	|\e e'(\ue)-\overline{Q'}|^2 \right.\\	
\left.	+2|\tfrac{1}{2} (\e^{{-p}}\pt \overline{\uea} +\e^{{p+1}}\nabla' \overline{\uet})  -\overline Q_{\alpha 3}|^2  + \lamm ( \ett(\overline{u}_{\e})  )^2 + (\ett(\overline{u}_{\e})-\overline{Q}_{33})^2 \Bigr]dx\right\}.
\end{multline}	
{
Taking the infimum over all sequences in~\eqref{2007231738}, observe that $\e^{-p}\int_{-1}^0\partial_3  u^{\e}_{\alpha} dx_3\wto 0$ weakly in $L^2(\o)$ for $-1<p<0$, and  $\e^{-p}\int_{-1}^0\partial_3  u^{\e}_{\alpha} dx_3\wto \zeta_{\alpha}(x')$ weakly in $L^2(\o)$
for $p=0$, and both $\e e'(\ue)$ as well $\e^{{p+1}}\nabla' \overline{\uet} \wto 0$ weakly in $L^2(\o)$, as proved in Proposition \ref{2009141558}-$g$).}
{Owing to the lower semicontinuity of all convex terms and 
using the characterisation of the set of  limit displacements (cf. Paragraph~\ref{2009141551}), we finally integrate with respect to the thickness in the film layer and read the energy in terms of the traces of displacements $(\zeta', \zeta_3)$ at the interface $\o\times \{0\}$.}
\newcommand{\bQat}{{\overline Q}_{\alpha 3}}
\begin{multline}
\Gamma\hbox{-}\liminf_{\e\to 0} \Jep (\ue) \geq {\frac{1}{2}}\int_\o 
\left( 
	|e'( \zeta')|^2-e'(\zeta')\nabla'\nabla'\zeta_3+ \frac{1}{3}|\nabla'\nabla'\zeta_3|^2 \right) dx'\\
	+{\frac{1}{2}}\int_\o \coefopt
	\left( \operatorname{tr}^2e'(\zeta') - \operatorname{tr}e'(\zeta')\Delta'\zeta_3 +\frac{1}{3}(\Delta'\zeta_3 )^2 \right)dx'\\
	+\inf_{\overline{Q}\in L^2(\o, \mathcal{Q}_B)} {\frac{1}{2}}\int_\o\left[ |{\overline{Q}}'|^2 + 2\left| \tfrac{1}{2}\zeta^{[p]}_\alpha(x')-\bQat \right|^2 + \lamm \zeta_3^2 + \left( \zeta_3-\overline{Q}_{33} \right)^2\right]dx'. 
\label{eqn:infnematic3d}
\end{multline}
Above, we use the short-hand notation 
${\zeta^{[p]}_\alpha}=\zeta'$ if $p=0$ and ${\zeta^{[p]}_\alpha}=0$ if $-1 < p <0 $.
Notice that in~\eqref{eqn:infnematic3d} we pass to infimum over $L^2(\o, \mathcal Q_B)$ because the integrand is independent of $x_3$. Finally the claim follows because
\begin{eqnarray}\label{eqn:lowerboundinterface}
	\inf_{\overline{Q}\in L^2(\o, \mathcal{Q}_B)}  \int_\o \left[|\overline{Q}'|^2 + 2\left| \tfrac{1}{2}(\zeta^{[p]})'(x')-\bQat \right|^2  + \left( \zeta_3-\overline{Q}_{33} \right)^2\right]dx'\equiv
 \int_{\o} \operatorname{dist}^2\left(\mathsf K\left(\zeta^{[p]}_\alpha,\zeta_3\right), \mathcal{Q}_B\right)\,dx',
\end{eqnarray}
which holds by virtue of the convexity of the set $\mathcal{Q}_B$.
\end{proof}

\begin{Remark}
Observe that the energy \eqref{eqn:lowerboundinterface}
 is written in terms of the trace of displacements at the common interface $\omega \times \{0\}$, which is necessarily well defined by the limits from above (in the film) and below (in the nematic layer), owing to the compactness of displacements.
In-plane and out-of-plane terms are coupled through cross products between the first in-plane derivatives of in-plane displacements and the second in-planes derivatives of the transverse component.
\end{Remark}

Below we prove the upper-bound inequality in the uniaxial case $X=Fr$. The discussion of the remaining cases $X=U$ or $X=B$ follows as a corollary and is 
{discussed in Remark~\ref{rem:biaxialcase}.}

\begin{proposition}\label{prop:glimsupp0}[Upper bound inequality, $-1 < p\leq 0$]
\label{prop:ub-allp}
Let $J^p_\e$ as in~\eqref{eqn:total3denergy} {with $X=Fr$}.
For every $u^\star\in KL$, there exists a sequence $(\ve)\subset L^2(\O, \R^3)$ such that $\ve \to u^\star=(\zeta' - x_3 \nabla'\zeta_3, \zeta_3)$ strongly in $L^2(\Of, \R^3)$ 
and
\begin{multline}
\frac{1}{2}\int_\o 
	 \left( |e'(\zeta')|^2 - e'(\zeta')\nabla'\nabla' \zeta_3 + \frac{1}{3}|\nabla'\nabla' \zeta_3|^2 \right)dx'
	  + \frac{1}{2}\int_\o
	 \coefopt \left(\operatorname{tr}^2 e'(\zeta') - \operatorname{tr} e'(\zeta') \Delta' \zeta_3 + \frac{1}{3}(\Delta' \zeta_3)^2 \right) dx'
\\
+\inf_{\overline Q\in  L^2(\o,\mathcal{Q}_B)}\frac{1}{2}\int_{\o} \left(   |\overline Q' |^2
+ 2|{\tfrac{1}{2}(\zeta^{[p]})'}-  \overline Q_{\alpha 3}|^2 + (\zeta_3-\overline Q_{33})^2
	+ \frac{ 1}{1-2\nu }\zeta_3^2\right)dx'
\ge	\limsup_{\e\to 0} J^p_\e(\ve)
\end{multline}
where we write
$({\zeta}^{[p]})' = \zeta'$
if $p=0$
and 
$({\zeta^{[p]}})' \equiv 0$
if $p\in (-1,0).$
\end{proposition}

The strategy is to decompose $\Ob$ into a finite partition of (columnar) grains so that $\Ob=\bigcup_j^m A_j$ up to a set of measure zero, and construct the recovery sequence for displacements and tensors on each individual grain.
Then, glueing individual grains will be performed after showing that boundary layer error terms can be made as small as desired.
{The proof follows with suitable modifications the one given, for a different scaling, in~\cite{cesana2018variational}.
For the readers' convenience,
 a self contained proof is given in Appendix.}

 \begin{Remark}
 \label{rem:biaxialcase}
{ The proof of the upper-bound inequality for $X=B$ follows
with simple modifications from the case $X=Fr$. Observe that it is not necessary to introduce weakly converging sequences of order tensors $Q^{\eta}$ nor a mollified $Q^{\eta,\delta}$. In fact, it is enough to approximate any order tensor in $L^2(\Ob,\mathcal{Q}_B)$ with $H^1(\Ob,\mathcal{Q}_B)$ tensors as done in \cite[Lemma 4.3]{cesana2018variational}
Then the proof in the case $X=U$ follows automatically thanks to the set inclusion $\mathcal{Q}_{Fr}\subset \mathcal{Q}_{U}\subset \mathcal{Q}_{B}$.
}
 \end{Remark}

\subsubsection{
Decoupled representation for shear-free plates ($-1<p<0$)
}\label{2012262313}
In order to read the result in the thick plate regime ($-1<p<0$) we perform a change of variable to decouple membrane from flexural deformations. 
Indeed, the peculiar structure of limit KL-displacements (cf.~\eqref{eqn:klsharp}) can be further exploited in the case at hand, where the nematic foundation is active only against transverse displacements, to represent the effective energy as a function of the traces of displacements at the mid-section of the film.

\begin{proof}[Proof of Theorem~\ref{thm:relaxweaklycoupled}]

Proposition
\ref{2012031723}
($\Gamma$-liminf inequality)
and 
Proposition
\ref{prop:glimsupp0}
($\Gamma$-limsup inequality)
show that, for $u\in KL$, we have
\begin{multline}\label{2012262300}
J_{}^-(u)=
  \int_\o \frac{1}{2}
\left( 	|e'( \zeta')|^2-e'(\zeta')\nabla'\nabla'\zeta_3+ \frac{1}{3}|\nabla'\nabla'\zeta_3|^2 \right) dx'
	\\
	+\frac{1}{2}\int_\o\coefopt
	\left( \operatorname{tr}^2e'(\zeta') - \operatorname{tr}e'(\zeta')\Delta'\zeta_3 +\frac{1}{3}(\Delta'\zeta_3 )^2 \right) dx'
	+
\frac{1}{2}\int_{\o} \operatorname{dist}^2\left(\mathsf K\left(0,\zeta_3\right), \mathcal{Q}_B\right)\,dx'.
\end{multline}
To write the $\Gamma$-limit result in $KL^{\sharp}$ we replace
$
\zeta'(x') = \zeta'_\sharp(x') +\frac{1}{2}\nabla' \zeta_3(x')$ for a.e. $x'\in \o$,
so that, after straightforward algebraic manipulations, \eqref{2012262300} yields
\begin{multline}
J_{}^-(u)= \frac{1}{2} \int_\o 
	\left( |e'( \zeta'_\sharp)|^2
	+\frac{1}{12}|{\nabla'\nabla' \zeta_3}|^2 \right) dx' \\
	+\frac{1}{2} \int_\o\coefopt
	 \left( \operatorname{tr}^2e'(\zeta'_\sharp) +\frac{1}{12}(\Delta' \zeta_3 )^2\right) dx'
	+
\frac{1}{2}\int_{\o} \operatorname{dist}^2\left(\mathsf K\left(0,\zeta_3\right), \mathcal{Q}_B\right)\,dx',
\end{multline}
for $u\in KL^{\sharp}$
and therefore
 Theorem~\ref{thm:relaxweaklycoupled} is proven.
\end{proof}

\section{Actuation}
\label{sec:actuation}

In this section we analyse the asymptotic models 
of nematic elastomer bilayers 
in the thin and thick plate regimes,
where the LC curvature energy blows up
{(see \cite{de-simone1993energy}
and \cite{des95})}.
In this limit, the LC orientation (as well as order states) is frozen and can be controlled by means of external forces and boundary conditions.
We label these problems of Actuation because tuning of the order tensor $Q$ leads to spontaneous shape morphing.
{As a paradigm for externally-controlled shape morphing, we perform the analysis of NLCE bilayers under an external electric field.}
From the mathematical standpoint, we deal with the limit as $\e\to 0$ and {$\delta_\e\to\infty$}, corresponding to the processes of structural relaxation for constant $Q$ with no optic relaxation.
Additionally, we require $\delta^2_\e \e^{p+2}\to\infty$. This corresponds to the limit regime of thin elastic foundations of small size. In this way, we model small NLCE units as building blocks of  
structures with heterogeneously patterned LC orientations, a proxy to non-isometric origami or  optically active  membranes
\cite{plucinsky2016programming}, \cite{plucinsky2018patterning}.

In presence of an electric field, the complete form of energy as introduced in \eqref{2009051223} is
\begin{multline}
\label{2009051225}
J^p_\e
(v,Q,\phi)=
 J^\e_f(v)+J^\e_b(v,Q)-J^\e_{ele}(Q,\phi)=\\
\frac{1}{2}\int_\Of \left( |\eab(v)|^2+({\e^{-2}}\ett(v))^2 + 2|\e^{-1}\eat(v)|^2\right)dx  + \lamm \left((\eaa(v))+ \e^{-2}\ett(v) \right)^2dx\\
	+\frac{1}{2} \int_\Ob  \left( |\e \eab(v)-\Qab|^2+ 
	\left( \ett(v)-\Qtt \right) ^2 \right) dx\\
	+ {\frac{1}{2}}\int_\Ob\left[ 2\left| \frac{1}{2} \left( \e^{p+1} \pa \vt + \e^{-p}\pt \va \right) -\Qat \right|^2+
	\lamm \left( \e\tr(e(v))+ \ett(v) \right)^2\right]dx\\
	+ \frac{1}{2}\int_\Ob \delta_\e^2 \left(\e^{2p+2}|\nabla' Q|^2+\left| {\pt Q} \right|^2  \right)dx
	-{\frac{1}{2}}\int_{\Ob}(\nabla^\e \phi)^T\mathsf D(Q)\nabla^\e \phi dx,
\end{multline}
where we have used the concise notation $\nabla^\e (\cdot):= \left(\nabla'(\cdot),\frac{1}{\e^{p+1}}\partial_3(\cdot)\right)$ to indicate the scaled gradient of a scalar function.

{
Due to the presence of an electrostatic field the sign of the energy~\eqref{2009051225} is undefined, resulting in a saddle structure for $J^p_\e$. 
The analysis of equilibrium points of $ J^p_\e$ for fixed $\e$ as the solution of a min-max problem has been performed in~\cite{cesana2009strain-order}.
The main ideas (recalled below) consist in showing that the min-max problem can be replaced by a minimisation under the differential constraint given by Gauss law.
}
{
Exploiting this idea, the characterisation of equilibrium configurations for
$ J^p_\e$, for fixed $\e>0$ and $\delta_{\e}\to 0$,
is described in~\cite{Cphd09}.
}

{
In the present situation,  our strategy is as follows.
We first compute the effective reduced electrostatic energy by computing the 
$\Gamma$-limit of $J^{\e}_{ele}(Q,\phi)$ under  Gauss law (Section~\ref{2012291150}).
By observing that the limiting electrostatic work is a continuous perturbation to the energy of the entire system, we obtain the desired asymptotic result by summing up the \ respective limit contributions.

\paragraph{Dielectric tensor.}
To characterise the dielectric tensor explicitly, we write
\begin{equation}\label{2009150120}
\epsilon_0\mathsf D(Q):={\epsilon_o}\left(\frac{2\epsilon_{\perp}+\epsilon_{||}}{3}  I+(\epsilon_{||}-\epsilon_{\perp})Q \right).
\end{equation}
Constants appearing in (\ref{2009150120})
(including $\epsilon_0>0$)
 are defined in Table \eqref{2008061744} and represent dielectric parameters of the nematic liquid crystal. 
The main point here is that, for every $Q\in \mathcal{Q}_X$, 
with $X=Fr,U$ or $B$, $\mathsf D(Q)$ is a symmetric positive definite matrix.
Consequently, there exists a constant $C>0$ such that
\begin{equation}\label{eqn:ellipticity}
\frac{1}{C} |\xi|^2\le  \xi^T \mathsf D(Q)\xi \le  C |\xi|^2,\qquad\forall\xi\in\R^3.
\end{equation}
As a direct consequence of \eqref{eqn:ellipticity}, 
$\phi\mapsto -J^\e_{ele}(\cdot, Q)$ is a concave (and non-positive) functional 
and therefore the total energy is not bounded below.
Before proceeding with the analysis of 
\eqref{2009051225}, we 
elucidate on the admissible space of electrostatic potentials we envision in our experiments.

\begin{Remark}[Boundary conditions for $\phi$]\label{2008202327}
We define a function $\phi_0\in H^1(\Ob)$
such that $\partial_3\phi_0=0$ a.e. in $\Ob$, a subset $\partial_D\omega\subset\partial\omega$ with 
$\mathcal{H}^1(\partial_D\omega)>0$,
and $\partial_D\Omega:=\partial_D\omega\times[-1,0]$.
We take $\phi\in H^1(\Ob)$ equal to $\phi_0$ on $\partial_D\Omega$
(in the sense of traces) and we say $\phi-\phi_0\in H^1_D(\Ob)$
where 
\begin{eqnarray}\label{2011261216}
H^1_D(\Ob):=\{ f\in H^1(\Ob), f=0 \textrm{ on } \partial_D\Omega\}.
\end{eqnarray}
\end{Remark}

For fixed $\e$ and $\delta_{\e}>0$ analysis of critical points of 
\eqref{2009051225}
is pursued in \cite{cesana2009strain-order}. 
{We summarise here} the result.

\begin{proposition}\label{prop:gauss}
Fix $Q\in L^2(\Ob,\mathcal{Q}_X)$ where $X$ stands for either $Fr$, $U$ or $B$. Let $\e>0$ and fixed. Let $\mathsf D(Q)$ as defined in \eqref{2009150120} Let $\phi_0$ as in 
Remark \ref{2008202327}.
First, there exists a unique solution to
\begin{eqnarray}\label{2011261223}
\min_{\phi\in H^1_D(\Ob)+\phi_0}
	 \int_{\Ob}\left(\nabla^\e \phi\right)^T\mathsf D(Q)\nabla^\e \phi dx.
\end{eqnarray}
Equivalently, the minimiser of \eqref{2011261223} is the (unique) solution to the full 3D Gauss Law
\begin{eqnarray}\label{2012271840}
\displaystyle{ -\div_{\e}\left( \mathsf D(Q)
\nabla^\e \phi
\right)
=0 \quad\text{ in } H^{-1}(\Ob) },
\end{eqnarray}
where $\div_{\e}=( \frac{\partial}{\partial x_1}+\frac{\partial}{\partial x_2}+\frac{1}{\e^{p+1}}\frac{\partial}{\partial x_3})$.

Second. Label $\phi_Q$ the solution to \eqref{2011261223}
for the given $Q\in L^2(\Ob,\mathcal{Q}_X)$.  Take a sequence $\{Q_k\}\subset L^2(\Ob,\mathcal{Q}_X)$ such that $Q_k\to Q$ strongly in $L^2(\Ob,\R^{3\times 3})$ as $k\to\infty$. Then,
\begin{eqnarray}\label{2011261229}
\phi_{Q_k}\to\phi_Q \textrm{ strongly in } H^1(\Ob),
\end{eqnarray}
where $\phi_{Q_k}$ is the solution to $\eqref{2011261223}$ when $Q$ is replaced by $Q_k$.
Third,
\begin{equation}\label{2011261230}
J^\e_{ele}(Q_k,\phi_{Q_k})=
{\frac{1}{2}}\int_{\Ob}\left(\nabla^\e\phi_{Q_k}\right)^T\mathsf D(Q_k)
\nabla^\e\phi_{Q_k} dx
\to
{\frac{1}{2}}\int_{\Ob}\left(\nabla^\e\phi_{Q}\right)^T\mathsf D(Q)
\nabla^\e\phi_{Q} dx
=J^\e_{ele}(Q,\phi_{Q}),
\end{equation}
\end{proposition}

\begin{Remark}
Precisely, $\phi_Q $ is defined as an operator mapping $L^2(\Ob,\mathcal{Q}_X)\mapsto H^1(\Ob)$.
In this sense \eqref{2011261229} is a statement regarding the continuity of such operator
with respect to the strong $L^2(\Ob,\R^{3\times 3})$
topology of order tensors.
With some abuse of notation, we adopt the same symbol to indicate both the abstract operator 
$\phi_Q :L^2(\Ob,\mathcal{Q}_X)\to H^1(\Ob)$
as well as the   function obtained when mapping a fixed
$Q$ with the mapping  $\phi_Q$. 
\end{Remark}

\begin{proof}[{Sketch of the Proof of Proposition} \ref{prop:gauss}]
It is enough to see that, for fixed $Q\in L^2(\Ob,\mathcal{Q}_X)$, $\phi\to J^\e_{ele}(Q,\cdot)$ is coercive thanks to \eqref{eqn:ellipticity} and Poincar\'e inequality.
Thanks to   \eqref{2009150120}, $\phi\to J^\e_{ele}(Q,\cdot)$ is strictly convex
and hence weakly lower semicontinuous.
Therefore, the minimum in~\eqref{2011261223} is attained by a unique minimiser and
its characterization as the solution to the corresponding Euler-Lagrange equations~\eqref{2012271840} is a classical result for elliptic integrals.
Lastly,~\eqref{2011261229} and~\eqref{2011261230} follow from standard continuity properties,
{(see, e.g., the proof~\cite[Proposition 2.2]{cesana2009strain-order}}).
\end{proof}

\begin{proposition}[Theorem 2.1, \cite{cesana2009strain-order}]\label{prop:2009141418}
Let 
$J^p_\e(v,Q,\phi)$ as in \eqref{2009051225} and $\phi_0$ as in 
 Remark \eqref{2008202327} where $\e,\delta_{\e}>0$ are fixed. Then, $(u^*,Q^*,\phi^*)$ is a min-max point of 
$J^p_\e(v,Q,\phi)$ that is
\begin{eqnarray}\label{2011260115}
J^p_\e(u^*,Q^*,\phi^*)=\min_{
u\in \mathcal V,
Q\in H^1(\Ob,\mathcal{Q}_X)
}
\max_{\phi\in H^1_D(\Ob)+\phi_0}
J^p_\e(u,Q,\phi),
\end{eqnarray}
if and only if $(u^*,Q^*,\phi^*)$ is a solution to

\begin{eqnarray}\label{2011260132}
\min
\left\{ J^p_\e(u,Q,\phi_Q):\quad
u\in \mathcal V,
Q\in H^1(\Ob,\mathcal{Q}_X) \right\},
\end{eqnarray}
where $\phi_Q\in H^1_{D}(\Ob)+\phi_0$ solves
\begin{eqnarray}\label{2009051244}
\displaystyle{ -\div_{\e}\left( \mathsf D(Q)
\nabla^\e \phi
\right)
=0 \quad\text{ in } H^{-1}(\Ob) }.
\end{eqnarray}
\end{proposition}

\begin{proof}[{Sketch of the Proof of Proposition \ref{prop:2009141418}}]

Consider \eqref{2011260115}. Proposition~\ref{prop:gauss}
shows that the maximum problem in \eqref{2011260115}
has a unique solution, for given $Q\in L^2(\Ob,\mathcal{Q}_X)$, denoted by $\phi_Q$. Thanks to ellipticity \eqref{eqn:ellipticity}, 
\begin{eqnarray}\label{2011260122}
\max_{\phi\in H_D^1(\Ob)+\phi_0} -J^\e_{ele}(Q,\phi)\ge -\frac{C}{\e^{2p+2}}\|\nabla\phi_0\|_{L^2(\Ob)}^2.
\end{eqnarray}
Thanks to the continuity of $Q\mapsto J^\e_{ele}(Q,\phi_Q)$
in the strong $L^2(\Ob,\R^{3\times 3})$ topology
\eqref{2011261230} and the boundedness from below
\eqref{2011260122} 
it follows that the functional {
$ J^p_\e(u,Q,\phi_Q)$ is equal to
$\max_{\phi\in H^1_D(\Ob)+\phi_0}
 J^p_\e(u,Q,\phi)
$ and, 
}
is coercive and lower semicontinuous in the weak-$H^1(\O,\R^3)$ topology for $u$ and 
in the weak-$H^1(\Ob,\R^{3\times 3})$ topology for $Q$.
Therefore the claim follows
with $\phi^*:=\phi_{Q^*}$. To show
\eqref{2011260132} coincides with \eqref{2011260115}, observe that the unique solution  of
$\max_{\phi\in H^1_D(\Ob)+\phi_0}
 J^p_\e(u,Q,\phi)$
is characterised by 
\eqref{2009051244}
as shown in 
Proposition
\ref{prop:gauss} (see Eqs.  \eqref{2011261223} and \eqref{2012271840}).
\end{proof}

{
To compute the asymptotics of the electrostatic work}
 we identify a class of dielectrics which we call ``nearly homogeneous''  materials (or regular, that is, non-singular) {in the transverse direction}. These are materials whose dielectric tensor ---although varying over $\Ob$--- lies in a neighbourhood of its average controlled by the layer thickness. }
The regular character of the dielectric matrix is, in turn, a consequence of the strong convergence of optic tensors and the continuity of the dielectric matrix.

\begin{definition}[Nearly homogeneous dielectric tensor]\label{2011292351}
Let $\mathsf D_{\varepsilon}\subset L^{\infty}(\Ob,\R^{3\times 3})$ for every $\varepsilon$ and symmetric and positive definite uniformly in $\e$, that is, there exists a universal constant $C>0$ such that
\begin{equation}\label{2012012318}
\frac{1}{C} |\xi|^2\le  \xi^T \mathsf D_{\e}(x)\xi \le  C |\xi|^2,\qquad\forall\xi\in\R^3; \textrm{ for a.e. } x\in \Ob; \forall\e>0.
\end{equation}
We define a nearly homogeneous dielectric tensor {(in the transverse direction)} a matrix such that
\begin{equation}\label{2012031504}
\mathsf D_{\varepsilon}(x_1,x_2,x_3)=\overline{\mathsf D}(x_1,x_2)+z_{\varepsilon}\mathsf D^{\sim}(x_1,x_2,x_3)
\end{equation}
where $\overline{\mathsf D}(x_1,x_2):\int_{-1}^{0}\mathsf D_{\varepsilon}(x_1,x_2,x_3)dx_3$, 
$z_{\varepsilon}\mathsf D^{\sim}(x_1,x_2,x_3):=\mathsf D_{\varepsilon}(x_1,x_2,x_3)-\overline{\mathsf D}(x_1,x_2)$, $\|\mathsf D^{\sim}\|_{L^{\infty}(\Ob,\R^{3\times 3})}\leq M$
where $M$ does not depend on $\e$ and
 $z_{\e}\to 0$ when $\e\to 0$.

\end{definition}
{Intuitively, such nearly homogeneous materials are a generalisation of homogeneous materials in the following sense.}
Over a thin layer of thickness $\e$ (that is, the geometrical dimension which is asymptotically small) we admit oscillations $x_3\to \mathsf D(x',\cdot)$  $\e$-close to a constant matrix, so that no further small length scales are present. 
We will show that the behaviour of the dielectric tensor for our dimension reduction problem 
 responds precisely to   assumption~\eqref{2012031504}. 
Indeed, for nematic elastomers in the actuation configuration,   \eqref{2012031504} is   a \emph{consequence of the topology} for admissible minimising sequences of order tensors and not a true \emph{material restriction}.

From the functional point of view, observe that near-homogeneity is an assumption on the strong convergence of dielectric tensors in the sense that, for matrices specified in \eqref{2012031504} we have
\begin{equation}\label{2012031537}
\mathsf D_{\varepsilon}(x_1,x_2,x_3)\to \overline{\mathsf D}(x_1,x_2)\textrm{ strongly in } L^2(\Ob)\textrm{ as } \e\to 0.
\end{equation}
(and vice-versa). Importantly, the same does not hold for the weak convergence of matrices. 
Indeed,
\begin{equation}\label{2012031540}
\mathsf D_{\varepsilon}(x_1,x_2,x_3)\rightharpoonup \overline{\mathsf D}(x_1,x_2)\textrm{ weakly in } L^2(\Ob)\textrm{ as } \e\to 0
\end{equation}
does not imply \eqref{2012031504}.

We remind a useful property of elliptic integrals (without proof) which we employ in the following.

\begin{lemma}\label{2011271633}
Let $\{\mathsf D_k \}\subset L^2(\Ob,\R^{3\times 3})$ with $\mathsf D_k$ symmetric, uniformly bounded and positive definite (that is, 
$\mathsf D_k=\mathsf D_k^T$ and $\frac{1}{C}|\xi|^2\le \xi^T\mathsf D_k\xi \le C|\xi|^2$ for every $\xi\in\R^3$, for some $C>0$) and 
$\mathsf D_k\to \mathsf D$ strongly in $L^2(\Ob,\R^{3\times 3})$.
Let $\{f_k\}\subset L^2(\Ob)$ with $f_k\rightharpoonup f$ weakly in $L^2(\Ob,\R^3)$. Then
\begin{eqnarray}\label{}
\int_{\Ob}f^T\mathsf Df dx    \le \liminf_{k\to\infty}\int_{\Ob}f_k^T\mathsf D_kf_k  dx.
\end{eqnarray}
 \end{lemma}

 \subsection{Convergence of the electrostatic work for nearly transversely  homogeneous dielectric tensors}\label{2012281800}
\begin{lemma}\label{2009142323}
Let $\phi_0$ as in Remark \ref{2008202327}
and $\mathsf D_{\e}(x)$ and $\overline{\mathsf D}$ as in Definition  \eqref{2011292351}. Define
\begin{eqnarray}\label{2011301213}
I_{\e}(\phi):=\left\{ 
\begin{aligned}{}
&\displaystyle{\frac{1}{2}\int_{\Ob} \left(\nabla^\e \phi\right)^T\mathsf D_{\e}(x)\nabla^\e \phi dx}  &\textrm{ if }  \phi\in H^1_{D}(\Ob)+\phi_0  \\
&+\infty &\textrm{otherwise in } L^2(\Ob)
\end{aligned} \right.
\end{eqnarray}
Then, the $\Gamma$-limit of $I_{\e}$ in the strong $L^2(\Ob)$ topology as $\e\to 0$ is
\begin{eqnarray}
I_{0}(\overline{\phi}):=\left\{ \begin{aligned}{}
&{\frac{1}{2}}\displaystyle{\int_{\Ob} (\nabla'\overline{\phi})^T\,\overline{\mathsf B}(x') \nabla'\overline{\phi}  \, Bx'} & \textrm{if }  \overline{\phi}\in H^1_{D}(\o)+\phi_0 \\
&+\infty &\textrm{otherwise in } L^2(\o)
\end{aligned} \right.
\end{eqnarray}
where 
\begin{equation}\label{2011271852}
	\overline{\mathsf B}(x') = \left(
	\begin{matrix}
\displaystyle{\overline{\mathsf D}_{11}-\frac{\overline{\mathsf D}_{13}^2}{\overline{\mathsf D}_{33}}}	 & \displaystyle{\overline{\mathsf D}_{12}-\frac{\overline{\mathsf D}_{13}\,\overline{\mathsf D}_{23} }{\overline{\mathsf D}_{33}}}	 \\
\displaystyle{\overline{\mathsf D}_{12}-\frac{\overline{\mathsf D}_{13}\,\overline{\mathsf D}_{23} }{\overline{\mathsf D}_{33}}}	 & \displaystyle{\overline{\mathsf D}_{22}-\frac{\overline{\mathsf D}_{23}^2}{\overline{\mathsf D}_{33}}}
	\end{matrix}\right)(x')
	=\overline{ \mathsf D'}(x') +\overline {\mathsf B}_\textrm{sh}(x'),
\end{equation}
and $\overline{\mathsf D}_{ij}=\overline{\mathsf D}_{ij}(x')$ 
are components of    $\overline{\mathsf D}(x')$;
$\overline{\mathsf D'}(x')$ is the top-left $2\times 2$ submatrix of $\overline{\mathsf D}(x')$
and
$(\overline{\mathsf B}_\textrm{sh})_{\alpha \beta}=-
\frac{1}{\overline{\mathsf D}_{33}}
\overline{\mathsf D}_{\alpha 3}
\overline{\mathsf D}_{\beta 3}
$. 
\end{lemma}

\begin{proof}
We prove the statement in three steps, first, we show compactness of minimising sequences, second, we show the lower bound inequality, third we prove the upper bound 
inequality.

\paragraph{Compactness.}
Take an admissible minimising sequence $(\phi_\e)\subset L^2(\Ob)$ for which uniform boundedness of the energy $I_{\e}(\phi_{\e})\le C$ implies, thanks to
\eqref{eqn:ellipticity},
\begin{eqnarray}\label{2011301241}
\left\| \left(\nabla'\phi_{\e},\frac{1}{\e^{p+1}}\partial_3\phi_{\e}\right)  \right\|^2_{L^2(\Ob)} \le C;\quad \frac{1}{\e^{2p+2}}\left\| \left( \partial_3\phi_{\e}\right)  \right\|^2_{L^2(\Ob)} \le C,
\end{eqnarray}
which yields, thanks to Poincaré's inequality, that
\begin{eqnarray}\label{2011301242}
  \phi_{\e}\rightharpoonup \overline{\phi} \textrm{ weakly in }
  H^1(\Ob);\quad
 \partial_3\phi_{\e} \to 0 \textrm{ strongly in } L^2(\Ob); \quad
 \frac{1}{\e^{p+1}}\partial_3\phi_{\e}\rightharpoonup c_{}
\textrm{ weakly in } 
 L^2(\Ob).
\end{eqnarray}
This identifies the limit space 
\begin{eqnarray}\label{2012281054}
H^1_D(\o):=\{ \overline{\phi}\in H^1(\o), \overline{\phi}=0 \textrm{ on } \partial_D\o\}.
\end{eqnarray}

\paragraph{Gamma-liminf inequality.}

It is enough to consider sequences making the functional finite and uniformly bounded in $\e$. We write
\begin{multline}\label{2011301253}
C\ge\liminf_{\e\to 0} I_{\e}(\phi_{\e})
\ge  {\frac{1}{2}}\int_{\Ob} \left(\nabla'\overline{\phi},c\right)^T\overline{\mathsf D}(x')\left(\nabla'\overline\phi,c\right)dx \ge
    {\frac{1}{2}}\int_{\o} \left(\nabla'\overline{\phi},\overline{c}\right)^T\overline{\mathsf D}(x')\left(\nabla'\overline\phi,\overline{c}\right)dx,
\end{multline}
where $\overline{\mathsf D}(x')$ is the average of $\mathsf D(x)$ over the height; $\overline{\phi}$
and $c$ are the weak limits introduced above. We remark that the inequality above holds due to lower semicontinuity thanks to Lemma
\eqref{2011271633}
 because $\mathsf D_{\e}$ converges strongly to $\overline{\mathsf D}$ in $L^2({\Ob})$ according to Definition  \eqref{2011292351}.
The last inequality above follows from Jensen's inequality, where the only function which possibly depends on $x_3$ is $c$. Here $\overline{c}$ is the average of $c$ over $x_3.$
Then, 
\begin{eqnarray}\label{2011300133}
    \int_{\o} \left(\nabla'\overline{\phi},\overline{c}\right)^T\overline{\mathsf D}(x')\left(\nabla'\overline\phi,\overline{c}\right)dx \ge
    \int_{\o} \left(\nabla'\overline{\phi},\overline{c}^*\right)^T\overline{\mathsf D}(x')\left(\nabla'\overline\phi,\overline{c}^*\right)dx =
    \int_{\o} \left(\nabla'\overline{\phi}\right)^T\overline{\mathsf B}(x')\left(\nabla'\overline\phi\right)dx
\end{eqnarray}
where 
\begin{eqnarray}\label{2011300138}
\overline{c}^*=-\frac{\overline{\mathsf D}_{13}\partial_1\overline{\phi} +\overline{\mathsf D}_{23} \partial_2\overline{\phi} }{ \overline{\mathsf D}_{33} } (x')
\end{eqnarray}
has been obtained by pointwise minimisation of the transverse term in the integrand of~\eqref{2011300133}.

\paragraph{Gamma-limsup.}
 Consider a general ${\phi}\in H_D^1(\o)+\phi_0$,
Take ${\phi}_{\e,\eta}=\overline{\phi}+{\e^{p+1}}\overline{c}^*(x_3+1)\ast\rho_{\eta}$ 
where $\overline{c}^*$ is defined in \eqref{2011300138}.
Here $\rho_{\eta}$
 is the standard mollifier in $C^{\infty}_c(\o)$. 
 Notice that with this choice $\e^{p+1}\overline{c}^*(x_3+1)\ast
 \rho_{\eta}\in C^{\infty}_c(\o)\cap  C^{\infty}(\Ob)$ and ${\phi}_{\e,\eta}$ satisfies prescribed boundary conditions and
 ${\phi}_{\e,\eta}  \to \overline{\phi}$
 strongly in $L^2(\Ob)$ as $\e\to 0$, for a fixed $\eta>0$.
Plugging ${\phi}_{\e,\eta}$ into  $I_{\e}(\cdot)$
we have
\begin{multline}\label{2011300116}
I_{\e}({\phi}_{\e,\eta})=
{\frac{1}{2}}\int_{\Ob} \left(\nabla'\overline{\phi},\overline{c}^*\ast\rho_{\eta}\right)^T{\mathsf D}_{\e}(x)\left( \nabla'\overline{\phi},\overline{c}^*\ast\rho_{\eta}\right)dx
+\\
   \int_{\Ob} \left(\nabla'\overline{\phi},\overline{c}^*\ast\rho_{\eta}\right)^T{\mathsf D}_{\e}(x)\left( \e^{p+1}(x_3+1)\nabla'(\overline{c}^* \ast\rho_{\eta})  ,0\right)dx\\
+
 {\frac{1}{2}}\int_{\Ob} \left( \e^{p+1}(x_3+1)\nabla'(\overline{c}^* \ast\rho_{\eta}),0\right)^T{\mathsf D}_{\e}(x)\left( \e^{p+1}(x_3+1)\nabla'(\overline{c}^* \ast\rho_{\eta}) , 0\right)dx.
\end{multline}
We now discuss the three summands appearing on the right-hand side of \eqref{2011300116}.
First, observe
\begin{eqnarray}\label{2011300258}
   \int_{\Ob} \left(\nabla'\overline{\phi},\overline{c}^*\ast\rho_{\eta}\right)^T\mathsf D_{\e}(x)\left( \nabla'\overline{\phi},\overline{c}^*\ast\rho_{\eta}\right)dx\to
   \int_{\o} \left(\nabla'\overline{\phi},\overline{c}^* \right)^T\overline{\mathsf D}(x')\left( \nabla'\overline{\phi},\overline{c}^* \right)dx'
   =\int_{\o} \left(\nabla'\overline{\phi}\right)^T\overline{\mathsf B}(x') \nabla'\overline{\phi}dx',\nonumber
\end{eqnarray}
as both $\eta,\e\to 0$ since $\overline{c}^*\ast\rho_{\eta}\to\overline{c}^*$ strongly in $L^2(\o)$, 
$\mathsf D_{\e}\to\overline{\mathsf D} (x')$ strongly in $L^2(\Ob)$ with $\mathsf D_{\e}(x)$ uniformly bounded for every $\e$.
Second, observe,
  $|\nabla'(\overline{c}^* \ast\rho_{\eta})|=| \overline{c}^* \ast\nabla'\rho_{\eta} |\leq M  \eta^{-2}$ and therefore for fixed 
$\e>0$ there exists $\eta=\eta(\e)$     such that
\begin{multline}\label{2011300302}
\left|\int_{\Ob} \left(\nabla'\overline{\phi},\overline{c}^*\ast\rho_{\eta}\right)^T\mathsf D_{\e}(x)\left( \e^{p+1}(x_3+1)\nabla'(\overline{c}^* \ast\rho_{\eta})  ,0\right)dx\right|
\le\\ M \e^{p+1}
\|\nabla'\overline{\phi},\overline{c}^*\ast\rho_{\eta}\|_{L^2(\o)} 
\|\nabla'(\overline{c}^* \ast\rho_{\eta})  ,0\|_{L^2(\o)}
\le
M \e^{p+1}
\|\overline{c}^* \ast\nabla'\rho_{\eta}  \|_{L^2(\o)}\le
O(\e).
\end{multline}
Finally, consider 
\begin{eqnarray}\label{2011300240}
 \int_{\Ob} \left( \e^{p+1}(x_3+1)\nabla'(\overline{c}^* \ast\rho_{\eta}),0\right)^T\mathsf D_{\e}(x)\left( \e^{p+1}(x_3+1)\nabla'(\overline{c}^* \ast\rho_{\eta}) , 0\right)dx
\le
\e^{2p+2} M\| \nabla'(\overline{c}^* \ast\rho_{\eta})\|_{L^2(\o,\R^2) }^2\le O(\e)
\nonumber
\end{eqnarray}

Thus one can take the sequence
$
{\phi}_{\e,\eta(\e)}=\overline{\phi}+{\e^{p+1}}\overline{c}^*(x_3+1)\ast\rho_{\eta(\e)}
$ 
to read the result.
\end{proof}

\begin{Remark}\label{2012031148} 
Because of the ellipticity of the three-dimensional matrix $\mathsf D$, the effective matrix $\mathsf{\overline{B}}$ defined by Equation~\eqref{2011271852} is, in particular, symmetric and positive definite.
\end{Remark}

\subsection{Continuity of electrostatic work}\label{2012291150}
\begin{lemma}\label{2103101500}
Let $\phi_0$ as in Remark \ref{2008202327},
$\overline{Q}$ constant in $\Ob$
 and take a  sequence
$\{Q_k\}\subset H^1(\Ob,\mathcal{Q}_X)$ of uniformly bounded order tensors.
Define, for $\e>0$ and $k\in \N$
\begin{eqnarray}\label{2011301735}
I_{k,\e}(\phi):=\left\{ 
\begin{aligned}
&\displaystyle{{\frac{1}{2}}\int_{\Ob}} \left(\nabla^\e \phi\right)^T\mathsf D(Q_k)
 \nabla^\e \phi dx & & \text{ in } H^1_{D}(\Ob)+\phi_0 \\
&+\infty & & \textrm{ otherwise in }  L^2(\Ob).
\end{aligned} \right.
\end{eqnarray}
Let
$Q_k\to  \overline Q$ strongly in $L^2(\Ob,\R^{3\times 3})$ as $k\to\infty.$
Then, the $\Gamma$-limit of $I_{k,\e}$ in the strong $L^2(\Ob)$ topology as $\e\to 0$ and $k\to\infty$ is
\begin{eqnarray}\label{2012281808}
I_{\infty,0}(\overline{\phi}):=\left\{ 
\begin{array}{ll}
\displaystyle{\frac{1}{2}}\int_{\Ob} (\nabla'\overline{\phi})^T \,\overline{\mathsf B}(\overline{Q}) 
\nabla'\overline{\phi}  dx'  &
\textrm{ in }  H^1_{D}(\o)+\phi_0\\
+\infty & \textrm{ otherwise in }   L^2(\Ob),
\end{array} \right.
\end{eqnarray}
{where}
\begin{equation}\label{2011302359}
{	\overline{\mathsf B}(\overline{Q}) = \left(
	\begin{matrix}
\displaystyle{\overline{\mathsf D}_{11}(\overline{Q})-\frac{\overline{\mathsf D}_{13}^2(\overline{Q})}{\overline{\mathsf D}_{33}}}	 & \displaystyle{\overline{\mathsf D}_{12}(\overline{Q})-\frac{\overline{\mathsf D}_{13}(\overline{Q}) \overline{\mathsf D}_{23}(\overline{Q})}{\overline{\mathsf D}_{33}}}	 \\
\displaystyle{\overline{\mathsf D}_{12}(\overline{Q})-\frac{\overline{\mathsf D}_{13}(\overline{Q})\overline{\mathsf D}_{23} (\overline{Q})}{ \overline{\mathsf D}_{33}(\overline{Q})}}	 & \displaystyle{\overline{\mathsf D}_{22}(\overline{Q})-\frac{\overline{\mathsf D}_{23}^2(\overline{Q})}{\overline{\mathsf D}_{33}(\overline{Q})}}
	\end{matrix}\right).}
\end{equation}
Also, denoting by $\div_{\e}=( \frac{\partial}{\partial x_1}+\frac{\partial}{\partial x_2}+\frac{1}{\e^{p+1}}\frac{\partial}{\partial x_3})$ the rescaled divergence and by $\phi_{Q_k,\e}$ the solution to the 3D Gauss equation
\begin{equation}
\label{2008201137}
	\phi\in H^1_{D}(\Ob)+\phi_0: \qquad -\div_{\e}\left( \mathsf D(Q_k)\nabla^\e \phi\right)=0 \textrm{ in } H^{-1}(\Ob),
\end{equation}
we have 
\begin{eqnarray}\label{2011301900}
\phi_{Q_k,\e}\to \overline{\phi}_{\overline{Q}} \textrm{ strongly in }H^1(\Ob),
 \end{eqnarray}
with $\overline{\phi}_{\overline{Q}}\in H^1(\Ob)$ such that $\partial_3 \overline{\phi}_{\overline{Q}}=0$
in $(-1,0)$
(equivalently, $\overline{\phi}_{\overline{Q}}\in H^1(\omega)$ constantly extended along $x_3$)
 and  
\begin{eqnarray}\label{2011301901}
\frac{1}{\e^{p+1}}\partial_3\phi_{Q_k,\e}\to-\frac{\overline{\mathsf D}_{13}(\overline{Q})\partial_1\overline{\phi}_{\overline{Q}}+\overline{\mathsf D}_{23}(\overline{Q})\partial_2\overline{\phi}_{\overline{Q}}}{\overline{\mathsf D}_{33}(\overline{Q})}
\textrm{ strongly in }L^2(\Ob),
\label{eqn:limitdielectric}
 \end{eqnarray}
 where
and $\overline{\mathsf D}_{ij}(\overline{Q})$ are components of the matrix $\overline{\mathsf D}(\overline{Q})$ and
  $\overline{\phi}_{\overline{Q}}$
is a solution to the 2D Gauss Law
\begin{equation}\label{2011301902}
\overline{\phi}\in H^1_{D}(\omega)+\phi_0: \qquad
-\div'( \overline{\mathsf B}(\overline{Q})\nabla'\overline{\phi})=0 \quad  \textrm{ in } H^{-1}(\omega), 
\end{equation}
with $\overline {\mathsf B}$ according to~\eqref{2011271852}.
Additionally,
\begin{multline}\label{2011301740}
\min_{\phi\in H^1_D(\Ob)+\phi_0}
\int_{\Ob}
\left( \nabla^\e \phi \right)^T \mathsf D(Q_k)\left( \nabla^\e \phi \right)dx
=
\int_{\Ob}
\left( \nabla^\e \phi_{Q_k, \e} \right)^T \mathsf D(Q_k)\nabla^\e \phi_{Q_k, \e} dx
 \\
 \to
 \int_{\omega}
\nabla\overline{\phi}_{\overline{Q}}^T \mathsf {\overline B}(\overline{Q})\nabla\overline{\phi}_{\overline{Q}} dx'
=
\min_{\phi\in H^1_D(\o)+\phi_0}
 \int_{\omega}
\nabla'\overline{\phi}^T \mathsf {\overline B}(\overline{Q})\nabla'\overline{\phi} dx',
\end{multline}
as $k\to\infty$ and $\e\to 0$.
\end{lemma}

\begin{proof}
For fixed $\e$, 
it is not restrictive to assume that
there exists $k=k(\e)$ such that $\|\mathsf D(Q_k)-\mathsf D(Q)\|_{L^2(\Ob,\mathbb{R}^{3\times 3})}\le \e$, that is, $\mathsf  D(Q_k)$ is a nearly transversely homogeneous dielectric matrix. Therefore,
Lemma~\ref{2009142323}
  applies verbatim 
with $ k=k(\e)$.
Consequently, \eqref{2011301740} follows directly from the convergence of the minimum and minimiser of  \eqref{2011301735}
to the minimum and minimiser  of \eqref{2012281808}.
We are left with showing  \eqref{2011301900} and \eqref{2011301901}.
First, from \eqref{2011301740} one has that
\begin{eqnarray}\label{}
\int_{\Ob}
\left( \nabla'\phi_{Q_{k(\e)},\e},\frac{1}{\e^{p+1}}\partial_3\phi_{Q_{k(\e)},\e}\right)^T \mathsf D(Q_{k(\e)})\left( \nabla'\phi_{Q_{k(\e)},\e},\frac{1}{\e^{p+1}}\partial_3\phi_{Q_{k(\e)},\e}\right)dx
\le C\end{eqnarray}
and by equi-coercivity     there exists $\overline{\phi}^*\in H^1(\Ob)$ such that, as $\e\to 0$,
\begin{multline}\label{2011302155}
 \phi_{Q_{k(\e)},\e}
 \rightharpoonup \overline{\phi}^*
\textrm{ weakly in } H^1(\o),\quad 
 \partial_3\phi_{Q_{k(\e)},\e}
\to 0
\textrm{ strongly in } L^2(\Ob),\\ 
{ \frac{1}{\e^{p+1}}\partial_3\phi_{Q_{k(\e)},\e}
\rightharpoonup c^*}
\textrm{ weakly in } L^2(\Ob)
\end{multline}
up to a subsequence here not relabelled.
Thanks to the Fundamental Theorem of $\Gamma$-convergence, such a sub-sequence converges to the minimiser $\overline{\phi}_{\overline{Q}}$ of the right hand side of 
\eqref{2011301740} in the sense specified by 
the first two terms in \eqref{2011302155}. This uniquely determines  $\overline{\phi}^*\equiv  \overline{\phi}_{\overline{Q}}$.
We notice that, since the solution to both the $\e$-dependent problem and the $\Gamma$-limits are unique due to strict convexity, the convergence is indeed recovered for the entire $k$-sequence and it is not necessary to pass to subsequences.

In order to identify $c^*$, we derive the associated Euler equations and pass to the limit, exploiting convergences established so far. 
Consider a generic test function $\theta\in H^1(\o)+\phi_0$.
We have
\begin{eqnarray}\label{}
\int_{\Ob}
 ( \nabla'\phi_{Q_{k(\e)},\e},\frac{1}{\e^{p+1}}\partial_3\phi_{Q_{k(\e)},\e})^T \mathsf D(Q_{k(\e)}) ( \nabla'\theta,0 )dx
=0, \quad \forall \theta\in H^1(\o)+\phi_0,
\end{eqnarray}
and in the limit $\e\to 0$
\begin{eqnarray}\label{2011302249}
\int_{\Ob}
 ( \nabla'\overline{\phi}_{\overline{Q}}, c^*)^T \mathsf {\overline D}(\overline{Q}) ( \nabla'\theta,0 )dx
=0, \quad \forall \theta\in H^1(\o)+\phi_0.
\end{eqnarray}
(Notice we have replaced $\phi^*$ with $\overline{\phi}_{\overline{Q}}$ above  as they are   identified.)
Additionally, $\overline{\phi}_{\overline{Q}}$ is such that
\begin{eqnarray}\label{2011302250}
\int_{\o}
 ( \nabla'\overline{\phi}_{\overline{Q}})^T  \mathsf {\overline B}(\overline{Q})  \nabla'\theta dx
=0, \quad\forall \theta\in H^1(\o)+\phi_0
\end{eqnarray}
by minimality,
and we can map the integral to $\Ob$ by a constant extension of its argument along $x_3$.
Observe that, by relabeling
with $\tilde{c}$
the  right hand side of  \eqref{2011301901},
 we have the identity
\begin{eqnarray}\label{2011302251}
\int_{\Ob}
 ( \nabla'\overline{\phi}_{\overline{Q}}, \tilde{c} )^T  \mathsf{\overline D}(\overline{Q}) ( \nabla'\theta,0)dx
\equiv 
\int_{\Ob}
 ( \nabla'\overline{\phi}_{\overline{Q}})^T  \mathsf {\overline B}(\overline{Q}) ( \nabla'\theta)dx
=0 
\end{eqnarray}
Therefore
 \eqref{2011302249} and \eqref{2011302251}
coincide,
 and the last 
property
in \eqref{2011302155}  follows with $\tilde{c}\equiv c^*$.
Finally, to pass from the weak convergence to the strong convergence we consider again \eqref{2011301740}. 
Upon replacing $\mathsf D(Q_{k(\e)})$ with $\overline{\mathsf D}(\overline{Q})$
in the second integral in \eqref{2011301740}
we obtain, as $\e\to 0$,
\begin{eqnarray}\label{2011302143}
{\int_{\Ob}
\left(  \nabla^{\e} \phi_{Q_{k(\e)}, \e} \right)^T \overline{\mathsf D}(\overline{Q})\nabla^{\e} \phi_{Q_{k(\e)}, \e} \,dx
 \to
 \int_{\Ob}
(\nabla\overline{\phi}_{\overline{Q}},c^*)^T \overline{\mathsf D}(\overline{Q})(\nabla\overline{\phi}_{\overline{Q}},c^*) dx}
\end{eqnarray}
and, in turn,
\begin{eqnarray}\label{2011302145}
{\int_{\Ob}
\left( \nabla'\phi_{Q_{k(\e)},\e}-\nabla\overline{\phi}_{\overline{Q}},\frac{1}{\e^{p+1}}\partial_3\phi_{Q_{k(\e)},\e}-c^*\right)^T \mathsf {\overline  D}(\overline{Q})\left(\nabla'\phi_{Q_{k(\e)},\e}-\nabla\overline{\phi}_{\overline{Q}},\frac{1}{\e^{p+1}}\partial_3\phi_{Q_{k(\e)},\e}-c^*\right)dx\to 0.}\nonumber
\end{eqnarray}
Using the estimate for elliptic dielectric matrices \eqref{eqn:ellipticity} we have
\begin{eqnarray}\label{2011302256}
{
 \left\| \nabla'\phi_{Q_{k(\e)},\e}-\nabla\overline{\phi}_{\overline{Q}},\frac{1}{\e^{p+1}}\partial_3\phi_{Q_{k(\e)},\e}-c^*\right\|^2_{L^2(\Ob,\mathbb{R}^3)}\to 0,\quad \textrm{as }\e\to 0},
\end{eqnarray}
and  
\eqref{2011301900} and
\eqref{2011301901} are proven.

\end{proof}
\begin{Remark}
{The strong $L^2(\Ob,\R^{3\times 3})$-convergence of order tensors
is key to ensure the strong $H^1$-convergence of the electrostatic potential
 solving Gauss equation.
This is a consequence of the G-closure of elliptic operators under strong convergence of its coefficients, cf.~\cite{dal-maso1993an-introduction}.
An outstanding open problem is the characterization of the G-closure for elliptic operators of the form $-\div(\mathsf D(Q_k)\nabla\cdot)$  under the weak 
$L^2(\Ob,\R^{3\times 3})$-convergence of order tensors.}
\end{Remark}

\begin{Remark}[{Opto-electric effects in bilayer structures}]
{
In the wake of relaxation induced by the dimension reduction over $x_3$, the  limit system is described by an effective matrix of relaxed dielectric parameters $ \overline{ \mathsf B}$, cf. \eqref{2011271852}. 
Note that, by virtue of \eqref{eqn:limitdielectric}, the third component of the dielectric field is always zero. 
This is the regime of planar dielectric fields (by analogy to the elastic case).
As in \eqref{2011271852}, we decompose
$ \overline{ \mathsf B}(\overline{Q})=  \mathsf{\overline  D}'(\overline{Q}) + \overline{\mathsf B}_{\text{sh}}(\overline{Q})$
where $\overline{\mathsf D'}$ is the upper-left $2\times 2$ submatrix of $\overline{\mathsf B}(\overline{Q})$, and $\mathsf{\overline  B}_\text{sh}$ is a matrix constructed only with shear terms, namely,
}
\begin{equation}\label{2012011240}
{
	\mathsf {\overline D}' = \left(
	\begin{matrix}
\mathsf {\overline D_{11}}  & \mathsf {\overline D_{12}} 	 \\
\mathsf {\overline D_{12}}  & \mathsf {\overline D_{22}} 
	\end{matrix}\right)(\overline Q),\qquad
		\mathsf {\overline  B}_\text{sh} = -\frac{1}{ \mathsf {\overline D_{33}}}\left(
	\begin{matrix}
\mathsf {\overline D^2_{13}}  & \mathsf {\overline D_{13}}\mathsf {\overline D_{23}} 	 \\
\mathsf {\overline D_{13}}\mathsf {\overline D_{23}}  & \mathsf {\overline D^2_{23}} 
	\end{matrix}\right)(\overline Q).}
\end{equation}
{
The former of the matrices is the dielectric tensor that describes purely planar electric fields $\phi=\phi(x')$ albeit in 3D structures which cannot relax through dimension reduction.

The matrix $\overline{\mathsf B}$ coincides with $\overline{\mathsf D}'$ if and only if  $\mathsf {\overline D_{13}}=\mathsf {\overline D_{23}}=0$.
This circumstance occurs when the optical order states induced by the liquid crystal are either planar in the $(x_1, x_2)$-plane or antiplanar, parallel to the $x_3$ direction. 
In this particular scenario, the conditions of plane dielectric field (i.e. electric stress) and plane electric field (i.e. electric strain) collapse.
All other states involving sheared out-of-plane dielectric displacements induce relaxation of the dielectric matrix.
}

\end{Remark}

\subsection{Convergence of mechanical energy and electrostatic work}\label{2103021500}

Finally, we are in a position to discuss the global $\Gamma$-convergence of the total energy of the system composed of elastic bending energy of the tensor $Q$, the bulk mechanical energy in the nematic layer $J_b^\e(v,Q)$, 
the mechanical bulk energy in film layer $J_f^\e(v)$ and the electrostatic work
stemming from an external source $J^\e_{ele}(Q,\phi)$. 
The full asymptotic result follows readily by combining the $\Gamma$-convergence results of
the elastic energy for the bilayer structure
and by noticing that the electrostatic work is a continuous perturbation of the total energy, in the sense specified by Lemma~\ref{2103101500}.
 
{
Unlike in the relaxation Section~\ref{sec:relaxation}, 
the order tensor $Q$ is treated as an independent variable. This allows us to discuss parametric problems which are relevant for applications (cf. Paragraph~\ref{2012291000}).
}
For  the sake of conciseness
we present in detail the results for the fully coupled scenario, that of thin nematics ($p=0$), and discuss the thick nematic ($-1<p<0$) case, 
 with  simple modifications, at the end of the section.

  Below and in the remainder of this section,
we introduce parametrised sequences $\delta_{\e_j}\equiv\delta_j\to\infty$ and $\e_j\to 0$
(with $\delta_j^2\e_j^{p+2}\to\infty$)
indexed by $\mathbb{N}\ni j\to\infty$, adopting the short-hand notation
$u_j$ instead of $u_{\delta_j,\e_j}$ and $Q_j$ instead of $Q_{\delta_j,\e_j}$.

Consider $J^p_{\e}$   as in \eqref{2009051225}. 
To tackle the asymptotics of the mechanical and electrostatic problem, we compute the limit of 
\begin{eqnarray}\label{2009142331}
\Fpe(u,Q)=
\left\{
\begin{aligned}
&\displaystyle{\max_{\phi\in H_D^1(\Omega)+\phi_0} J^p_{\e}(u,Q,\phi)} & &\textrm{in } 
\mathcal{V} \times H^1(\Ob,\mathcal{Q}_{X})\\
&+\infty   & &\textrm{otherwise in }
L^{2}(\Omega,\R^3) \times L^2(\O,\R^{3\times 3}),
\end{aligned} \right.
\end{eqnarray}
where $X$ stands for either $Fr, U$ or $B$.
{In view of Proposition~\ref{prop:gauss}, the argument of the max is the solution to its 3D Gauss equation, for $\e>0$.
Crucially, the resulting functional~\eqref{2009142331} is coercive in $(u,Q)$, as indicated in Proposition~\ref{prop:2009141418}.
We have the following results.}

\begin{theorem}\label{thm:actuationcoupled}
Let  $\Fpe$ as in \eqref{2009142331}
with $p=0$.
We have
\begin{eqnarray}\label{2008262033}
\Gamma\hbox{-}\lim_{\e\to 0 }
\mathcal{F}_{\e}^{0}(\ukl,\overline Q)=
\mathcal{F}^{0}(\ukl,\overline{Q})
\end{eqnarray}
in the strong-$L^2(\Omega_f,\R^3)\times$strong-$H^1(\Ob,\R^{3\times 3})$  topology, where 
\begin{eqnarray}\label{}
\Fz(\ukl,\overline{Q})=
\left\{
\begin{array}{ccc}
\displaystyle{\max_{\overline{\phi}\in H_D^1(\o)+\phi_0}J^0(\ukl,\overline{Q},\overline{\phi})} &\textrm{ on } 
KL\times  \mathcal{Q}_X\\
+\infty  &\textrm{otherwise in }
L^2(\omega,\R^3)\times H^1(\omega,\R^{3\times 3}),
\end{array} \right.
\end{eqnarray}
where $\ukl=(\zeta'(x')-x_3\nabla'\zeta_3(x'),\zeta_3(x'))$, $\overline{Q}$ is a constant tensor
 and  
\begin{multline}\label{eqn:endenslimitactuationpzero}
J^0(\ukl,\overline{Q},\overline{\phi}) =	\displaystyle{ \frac{1}{2}\int_\o\left[ 
	 |e(\zeta')|^2 - e(\zeta')\nabla' \zeta_3 + \frac{1}{3}|\nabla'\nabla' \zeta_3|^2 + 
	 \coefopt \left((\tr (\zeta')^2- \tr (\zeta') \Delta' \zeta_3 + \frac{1}{3}(\Delta' \zeta_3)^2 \right)\right]dx' } \\
+ \displaystyle{	\frac{1}{2}\int_{\o}\left[ |\overline Q'|^2 + 2|\frac{1}{2}\zeta'- (\overline Q e_3)'|^2 + (\zeta_3-\overline Q_{33})^2+ \lamm\zeta_3^2 
\right]dx' }
-{\frac{1}{2}}\int_{\o}\nabla'\overline{\phi}^T \overline{\mathsf B}(\overline{Q})\nabla'\overline{\phi} dx'.
\end{multline}
\end{theorem}
\begin{proof}
First, we show $\Gamma$-convergence
{
of the mechanical energy alone, noticing that, by fixing $\phi_0\equiv 0$ 
}
$ \max_{\phi\in H_D^1(\Omega)} J^p_{\e}(u,Q,\phi)
=J^p_{\e}(u,Q,0)$.
\paragraph{Compactness.}
Taking sequences $(u_j,Q_j)$ such that {$\Fpej(u_j,Q_j)=J^p_{\e_j}(u_j,Q_j,0)\le C$}, 
where
 $C$ does  not depend on $\e_j$ nor $\delta_j$, shows that the limit set of displacements is given by $KL$
 (see Proposition \ref{2009141558}). To determine the limit of the order tensors, observe that 
$\delta_j^2\int_{\Ob}|\nabla Q_{j}|^2 dx\le J^p_{\e_j}(u_j,Q_j,0)\le C$
implies
$Q_j\to \overline{Q}$
strongly in $H^1(\Ob,\mathcal{Q}_X)$
where $X$ stands for either $Fr, U$ or $B$ and $\overline{Q}$
is necessarily constant.

\paragraph{Liminf inequality.}
The proof 
in the film layer is identical to the proof of 
Proposition \ref{2012031723}.
For the nematic layer the argument of Proposition \ref{2012031723} can be easily adapted to the present situation.

\paragraph{Limsup inequality.}

It is enough to take the trivial recovery sequence ${Q}_j\equiv \overline{Q}$ constant in $\Ob$ in  Proposition~\ref{prop:glimsupp0},
without the need of  boundary layers.

\paragraph{Gamma-convergence for general $\phi_0\in H^1(\o)$.}
Observe that, for $(u,Q)\in \mathcal{V} \times H^1(\Ob,\mathcal{Q}_{X})$
$$
\Fze(u,Q)=
J^0_{\e}(u,Q,0)
-\min_{\phi\in H^1_D(\Ob)+\phi_0}{\frac{1}{2}}\int_{\Ob}\nabla\phi^T\mathsf D(Q)\nabla\phi dx.
$$
In light of
Lemma \ref{2103101500}, 
$Q\mapsto \operatorname{max}_{\phi\in H_D^1(\Ob)+\phi_0}  [
-\int_{\Ob}\nabla\phi^T \mathsf D(Q)\nabla\phi dx]$
is a continuous perturbation of
$J^0_{\e}$ in 
 the strong $L^2(\Ob,\R^{3\times 3})$
topology and there follows
\begin{equation}\label{}
\min_{\phi\in H_D^1(\Ob)+\phi_0} \left[
\int_{\Ob}\nabla\phi^T \mathsf D(Q_k)\nabla\phi dx\right]
\to
\min_{\phi\in H_D^1(\o)+\phi_0} \left[
\int_{\o}\nabla'\overline{\phi}^T \overline{\mathsf B}(\overline{Q})\nabla'\overline{\phi} dx'\right],
\end{equation}
as $Q_k\to \overline{Q}$
strongly in $L^2(\Ob,\R^{3\times 3})$
for $Q_k\in\mathcal{Q}_X$, $\forall k\in\mathbb{N}$
and $\overline{Q}:\omega\to\mathcal{Q}_X$  is constant. Therefore the claim follows by a standard property of $\Gamma$-convergence ensuring stability with respect to continuous perturbations, cf.
\cite{dal-maso1993an-introduction}.

\end{proof}

Consider now $-1<p<0$.

\begin{theorem}\label{thm:actuationthick}
Let  $\Fpe$ as in \eqref{2009142331}
with $-1< p<0$.
We have
\begin{eqnarray}\label{2008262033}
\Gamma\hbox{-}\lim_{\e\to 0 }
\mathcal{F}_{\e}^{p}(\ukl,\overline Q)=
\mathcal{F}_{}^{-}(\ukl,\overline{Q})
\end{eqnarray}
in the strong-$L^2(\Omega_f,\R^3)\times$strong-$H^1(\Ob,\R^{3\times 3})$  topology, where
\begin{eqnarray}\label{}
\Fm(\ukl,\overline{Q})=
\left\{
\begin{array}{ccc}
\displaystyle{\max_{\overline{\phi}\in H_D^1(\o)+\phi_0} J^-(\ukl,\overline{Q},\overline{\phi})} &\textrm{ on } 
KL_\sharp \times  \mathcal{Q}_X\\
+\infty  &\textrm{otherwise in }
L^2(\omega,\R^3)\times H^1(\omega,\R^{3\times 3}),
\end{array} \right.
\end{eqnarray}
where $\ukl=(\zeta'_{\sharp}(x')-(x_3-\frac{1}{2})\nabla'\zeta_3(x'),\zeta_3(x'))$,  $\overline{Q}$ is a constant tensor
 and
\begin{multline}\label{eqn:endenslimitactuationpzero}
 J^-(\overline{ u},\overline{Q},\overline{\phi}) =	\displaystyle{ \frac{1}{2}\int_\o\left[ 
	 |e(\zeta'_{\sharp})|^2 - e(\zeta_{\sharp}')\nabla'\nabla' \zeta_3 + \frac{1}{3}|\nabla'\nabla' \zeta_3|^2 + 
	 \coefopt \left((\tr(\zeta'_{\sharp}))^2- \tr (\zeta'_{\sharp}) \Delta'\zeta_3 + \frac{1}{3}(\nabla'\nabla' \zeta_3)^2 \right)\right]dx' } 
\\
+ \displaystyle{	\frac{1}{2}\int_{\o}\left[ |\overline Q'|^2 + 2|(\overline Q e_3)'|^2 + (\zeta_3-\overline Q_{33})^2+ \lamm\zeta_3^2 
\right]dx' }
-{\frac{1}{2}}\int_{\o}\nabla'\overline{\phi}^T \overline{\mathsf B}(\overline{Q})\nabla'\overline{\phi} dx'.
\end{multline}
\end{theorem}

\begin{proof}
Follows as in Proof of Theorem \eqref{thm:actuationcoupled} with obvious modifications.
\end{proof}

\subsection{Physical implications}

{Convergence of minima and minimisers of
$\Fpe(u,Q)$   follows} easily from equicoercivity. Let $Q \in
H^1(\Omega,\mathcal{Q}_X)$. By minimality and $(\ref{2009150120})$ we
have
\begin{eqnarray}\label{cont0200}
 \inf_{\phi\in
H^{1}_{D}(\Ob)+\phi_0}
\int_{\Ob} \nabla\phi^T \mathsf D(Q)\nabla\phi
dx
\le
M\int_{\Ob}|\nabla\phi_0|^2dx=2C,
\end{eqnarray}
for some $M>0$.
Now we can write, 
for every $(u,Q)\in 
H^1(\Omega,\R^3)\times H^1(\Ob,\mathcal{Q}_X)$, where $X$ stands either for $Fr,U$ or $B$,
\begin{gather}\label{cont0206}
\Fpe(u,Q)
=
\max_{\phi\in H^1_D(\Ob)+\phi_0}{J}^p_{\e}(u,Q,\phi) 
  \geq
 {J}^p_{\e}(u,Q,\phi_0)=
J^p_{\e}(u,Q,0)- C,
\end{gather}
(notice that constants appearing in \eqref{cont0200} and \eqref{cont0206} are equal)
and hence equicoercivity is obtained in $H^1(\Omega,\mathcal{Q}_X)\times
H^1_{}(\Omega,\R^3)$ by applying Korn's and Poincar\'{e}
inequality and considering that $\mathcal{Q}_X$ is a bounded
set.

As a direct consequence, we obtain the following   standard result
(see \cite{dal-maso1993an-introduction}).
{For conciseness, we tacitly assume that minimisation is performed for all free variables whenever the minimisation argument is not apparent.}

\begin{corollary}
\label{2009151414}
Consider
$\Fpe$ and $\Fp$
as defined in
Theorem \ref{thm:actuationcoupled}.
Then:
\begin{eqnarray}
\min_{
}
\Fp =\lim_{j\to+\infty} \left( 
\min_{
}\mathcal{F}_{\e_j}^{p}
\right) \ \qquad \emph{(convergence of minima)}.\nonumber
\end{eqnarray}
Let $\{u_j,Q_j\}\subset L^2(\Omega,\R^3)\times H^{1}(\Ob,\R^{3\times
3})$ be a minimising sequence for
$\Fpe$ (i.e.
$\lim_{j\to \infty }\Fp_{\e_j}(u_j, Q_j)=\lim_j\inf\Fp_{\e_j}$).
Then, up to a subsequence (not relabelled), 
$u_{j} \wto \ukl$,
$Q_{j} \to\overline{Q}$
with
$\ukl=(\zeta'-x_3\nabla'\zeta_3,\zeta_3)$;
$ \zeta'\in H^1(\o,\R^2)$, $\zeta_3\in H^2(\o)$,
with constant
$\overline{Q}\in \mathcal{Q}_X$, {then}
\begin{eqnarray}
\Fp(\ukl, \overline{Q}) =
\min_{
}\Fp\qquad
\emph{(convergence of minimum points)}.\nonumber
\end{eqnarray}
\end{corollary}

\subsubsection{Convergence of saddle-points }

Theorem~\ref{thm:actuationcoupled}
{implies convergence of equilibrium configurations for asymptotic models of nematic elastomer bilayers under electric fields.} 
We make this explicit.
\begin{corollary}[Convergence of min-max problems]\label{2012291436}

Consider $\Fze$ 
as  defined in \eqref{2009142331}
 and $\Jg$
as in
Theorem \ref{thm:actuationcoupled}.
Then we have (here $X$ stands either for
$Fr,U$ or $B$).
\begin{enumerate}
\item \emph{(Convergence of min-max values)}
\end{enumerate}
\begin{eqnarray}
\min_{
\begin{array}{c}
\scriptstyle (\ukl,\overline{Q})\in  KL\times
\scriptstyle  \mathcal{Q}_X
\end{array}
}
\left(
\max_{\phi\in
H^1_D(\omega)+\phi_0}
\Jg(\ukl,\overline{Q},\overline{\phi})
\right)
=\lim_{j\to\infty} \left(
\inf_{
\begin{array}{c}
\scriptstyle (u,Q)\in \mathcal V\times \scriptstyle   H^1(\Ob,\mathcal{Q}_X)
\end{array}
}
\max_{\phi\in H^1_{D}(\Ob)+\phi_0}
J_{\e_j}^{0}(u,Q,\phi)
 \right).\nonumber
\end{eqnarray}
This is equivalent to
\begin{eqnarray}
\min_{
\begin{array}{c}
\scriptstyle (\ukl,\overline{Q})\in  KL\times
\scriptstyle  \mathcal{Q}_X
\end{array}
}
\left(\Jg(\ukl,\overline{Q} ,\overline{\phi} ) \emph{ sub 2D Gauss Law }
(\ref{2011301902})  \right)
\qquad \nonumber\\
=\lim_{j\to+\infty} \inf_{
\begin{array}{c}
\scriptstyle (u,Q)\in
\mathcal V \times
\scriptstyle  H^1(\Ob,\mathcal{Q}_X)
\end{array}
}\left( J_{\e_j}^{0}(u,Q,\phi)\emph{ sub 3D Gauss Law }
(\ref{2008201137}) \right).\nonumber
\end{eqnarray}
Denote by $\phi_{Q}$ the solution to the 3D Gauss equation
$(\ref{2008201137}$) for some $Q\in H^1(\Ob,\mathcal{Q}_X)$. Let $\{u_j,Q_j,\phi_{Q_j}\}\subset
\mathcal V\times
H^1(\Ob,\mathcal{Q}_X)
\times
H^1_{D}(\Ob)+\phi_0$ be a min-maximising sequence for
$\{J^0_{\e_j}\}$, i.e.
\begin{displaymath}
\lim_{j\to+\infty}
J_{\e_j}^{0}\bigl(u_j,Q_j,\phi_{Q_j}\bigr)=
\lim_{j\to+\infty}\left(\inf_{
\begin{array}{c}
\scriptstyle (u,Q)\in
\mathcal V
\scriptstyle   \times H^1(\Ob,\mathcal{Q}_X)
\end{array}
}
\max_{\phi\in
H^1_{D}(\Ob)+\phi_0}
J_{\e_j}^{0}(u, Q,\phi)\right),
\end{displaymath}
or, equivalently,
\begin{eqnarray}
\lim_{j\to+\infty}J_{\e_j}^{0}\bigl(u_j, Q_j,\phi_{Q_j}\bigr)= 
\lim_{j\to+\infty}
\inf_{
\begin{array}{c}
\scriptstyle (u,Q)\in  \mathcal V \times H^1(\Ob,\mathcal{Q}_X)
\end{array}
}
\Bigl(J_{\e_j}^{0}\bigl(u,Q,\phi\bigr)\emph{ sub 3D Gauss Law
} (\ref{2008201137})\Bigr).\nonumber
\end{eqnarray}

Then, as $j\to\infty$ and up to a subsequence (not relabelled), 
$u_{j} \wto \uklsad\in KL$
weakly   in
$H^1(\Omega_f,\R^3)$
where ${\uklsad}=((\zeta^*)'-x_3\nabla'\zeta_3^*,\zeta_3^*)$; $ (\zeta^*)'\in H^1(\o,\R^2)$, $\zeta_3^*\in H^2(\o)$;
$Q_{j} \to \Qsad$
strongly in $H^1(\Ob,\R^{3\times 3})$
 with
$\Qsad\in\mathcal{Q}_X$ constant; 
and $\phi_{ Q_j}
\to\overline{\phi}_{\Qsad}$ strongly in $H^1(\Omega_b)$ and
\begin{enumerate}
\item[2.]
\emph{(Convergence of min-max points)}
\end{enumerate} 
\begin{displaymath}
\Jg (\uklsad, \Qsad,\phisad)
=
\min_{
\begin{array}{c}
\scriptstyle (\ukl,\overline{Q})\in KL\times 
\scriptstyle \mathcal{Q}_X
\end{array}
}
\left(
\max_{\overline{\phi}\in H^1_{D}(\omega)+\phi_0}
\Jg (\ukl, \overline{Q},\overline{\phi})
\right)
\end{displaymath}
or, equivalently,
\begin{eqnarray}
\Jg (\uklsad, \Qsad,\phisad)
=
\min_{
\begin{array}{c}
\scriptstyle (\ukl,\overline{Q})\in KL\times 
\scriptstyle   \mathcal{Q}_X
\end{array}
}
\left(
\Jg (\ukl, \overline{Q},\overline{\phi})
\textrm{ sub 2D Gauss Law }
(\ref{2011301902})\right).\nonumber
\end{eqnarray}
\end{corollary}

\begin{proof}
 The results above follow from Theorem \ref{2009151414} in light of 
Proposition \ref{prop:2009141418}.

 \end{proof}

As a consequence of convergence of saddle points, we infer that the saddle structure is preserved in the limit problem, thus equilibrium in the limit system is given by min-max points.

\begin{Remark}
Theorems \eqref{2009151414} and \eqref{2012291436}
have an analogue for the regime $-1<p<0$ (not displayed here).
\end{Remark}

 \subsubsection{Application to the mechanical actuation of the director}\label{2012291000}

Results of the previous section still apply when minimisation over the pair $(u,Q)$
is replaced with 
{a \emph{parametric}}
minimisation over the displacement $u$ only and for a given matrix $Q\in\mathcal{Q}_X$.
The following lemma describes the situation where the order tensor $Q$ is \emph{frozen}, that is, is considered as imposed by an external field (not necessarily electric) and not subject to minimisation.
This problem corresponds to determining the spontaneous deformation and shape change of bilayer structures when the liquid crystal order tensor is regarded as a load parameter.
The purpose is to highlight two mechanisms.
When the tensor describes perfect alignment of liquid crystal molecules with a distinguished optical axis, that is $Q \in \mathcal{Q}_{Fr}$, minimisation represents the controlled shape change of a bilayer driven by collective reorientation of molecules.
Contrarily, in conceptual experiments where the $Q$ tensor is taken in the set $Q_U$ or $Q_B$, low order sates, such as optical isotropy, biaxial states, and order melting also are admissible. In this case falls the thermal actuation of plates, when one controls separately the director and optical axis (represented by the eigenframe of the tensor $Q\in \mathcal{Q}_U$) and the degree of order of nematic molecules (represented by the eigenvalues of $Q\in \mathcal{Q}_U$), see~\cite{plucinsky2016programming,plucinsky2018patterning},

\begin{theorem}\label{lem:existenceactuation}
Let  $\Fz$ as in Theorem \eqref{thm:actuationcoupled}.
Fix $\overline{Q}\in \mathcal{Q}_X$, where $X$ stands for either $Fr,U$ or $B$ and assume $\varphi_0\equiv 0.$
Then there exists a unique solution to
\begin{eqnarray}\label{}
\min_{u\in L^2(\o,\R^3)} \Fz(\ukl,\overline Q). 
\end{eqnarray}
\end{theorem}

\begin{proof}
This follows from an application of the direct method in the calculus of variations. 
{It is enough to consider
displacements that make the energy finite. Take $\ukl=(\zeta'(x')-x_3\nabla'\zeta_3(x'),\zeta_3(x'))$
with $\zeta'\in H^1(\omega, \R^2),   \zeta_3\in H^2(\omega)$ and define 
$\mathcal{E}(\zeta',\zeta_3):=\Fz(\ukl,\overline Q)$.
Taking a minimising sequence $(\zeta',\zeta_3)_k\in H^1(\o,\R^2)\times H^2(\o)$ for every $k\in\N$
such that $\mathcal{E}((\zeta)'_k,(\zeta_3)_k)\leq C$, it follows that}
$$
||e'(\zeta')_k||^2_{L^2(\o, \R^{2\times 2})}+||\nabla^2(\zeta_3)_k||^2_{L^2(\o, \R^{2\times 2})}
+||(\zeta')_k||^2_{L^2(\o, \R^2)}+||(\zeta_3)_k||^2_{L^2(\o)}\leq C,
\quad\forall k\in\N.$$
By invoking Poincaré and Korn inequalities, along the transverse direction  and for the in-plane symmetrised gradient respectively,
we have
$
( \zeta')_k\rightharpoonup 
\zeta'$ weakly in $H^1(\o,\R^2)$ and $(\zeta_3)_k\rightharpoonup  \zeta_3$ weakly in $H^2(\o)$,
for some $( \zeta',\zeta_3)\in H^1(\o,\R^2)\times H^2(\o)$. Then, by convexity, $\mathcal{E}(\zeta',\zeta_3)$ is weakly lower semicontinuous and therefore the claim follows.
\end{proof}
\begin{Remark}
Theorem~\ref{lem:existenceactuation}
has an analogue for the regime $-1<p<0$, which is a consequence of Theorem \ref{thm:actuationthick},
whose proof follows with obvious modifications.
\end{Remark}

\paragraph{Numerical example of Fig. \ref{fig:actuation}.}

To illustrate the purpose of the analysis so far performed, we have presented in Figure~\ref{fig:actuation} a numerical actuation experiment for a thin nematic bilayer membrane which exemplifies a nontrivial solution of an actuation mechanism performed on the basis of simple ingredients.
We are interested in inducing out-of-plane displacements via nematic actuation, and, through coupling between membrane deformations and bending modes, possibly exert work.
Consider the square domain $\omega = (0, 1)^2$ clamped a the boundaries and subject to an imposed (frozen) director $Q_0 = n_0\otimes n_0 -\frac{1}{3}I_3$ where $n_0 = (e_1 + e_3)/\sqrt{2}$, as displayed in the cartoon in Figure~\ref{fig:actuation}.left.
Our computation refers to the fully coupled model of Theorems~\ref{thm:actuationcoupled}
and
\ref{lem:existenceactuation}, where nematic actuation directly activates a spontaneous in-plane stretch and transverse displacements.
The (unique) equilibrium configuration, cf. Theorem~\ref{lem:existenceactuation}, displays a non-symmetric bending mode coupled to planar membrane deformations, in competition with homogeneous Dirichlet-type boundary conditions on $\partial \o$. The spontaneous stretch is triggered by the strong opto-elastic strain coupling which characterises the nematic active layer in the actuation regime.

In Figure~\ref{fig:actuation}-right we plot the value of (the norm of) in-plane displacements $|\zeta'|$, in the deformed configuration, with a discrete colour coding for readability.
Note that the explicit coupling between the in-plane and out-of-plane deformation is due to the cross term in Equation~\eqref{eqn:endenslimitactuationpzero}, resulting in an out-of-plane deflection above the reference $z=0$ surface.
In addition, because the actuator field is tilted with respect to the azimuthal axis, both shear and vertical terms of the active foundation are effective.
The numerical solution has been obtained by finite elements discretisation in the FEniCS environment~\cite{logg2012automated} using PETSc~\cite{petsc-efficient,petsc-user-ref} as data management and linear algebra package.

\paragraph{Acknowledgements.}
PC's work is supported by JSPS KAKENHI Innovative Area Grant Number JP19H05131.
PC holds an honorary appointment at La Trobe University and is a member of GNAMPA.
The authors are grateful to Epifanio Virga for his advice and comments.

\printbibliography

 \section{Appendix}
\label{sec:appendix}

{
Before showing the proof of Proposition \ref{prop:glimsupp0}, we need to introduce a collection of auxiliary results.}

\begin{lemma}[]
\label{lem:m3asfacts}
Let $\{A_j\}_{j=1}^m$ be a finite collection of   domains of the form
$\o_j\times (-1,0)$, where $\o_j\subset\o$
are
  open and bounded sets.
For any $\overline{Q}\in L^2(A_j,\mathcal{Q}_B)$ and constant, 
there exists
\begin{enumerate}
\item a sequence 
$(Q^{\eta})
\subset L^2(A_j, \mathcal{Q}_{Fr})$ of piecewise constant
tensors parameterised by $\eta>0$
such that
$Q^{\eta}(x)\rightharpoonup \overline{Q}$
weakly as $\eta\to 0$
 with $Q^{\eta}(x)\in \mathcal{Q}_{Fr}$ for every $\eta$
for a.e.  $x\in A_j$;
\item a sequence  $(Q^{\eta,\delta}) \subset C^{\infty}(A_j,\mathcal{Q}_{Fr})$  such that $Q^{\eta,\delta}(x)\rightharpoonup \overline{Q}$ weakly
in $L^2(\Ob,\R^{3\times 3})$
 as $\eta,\delta\to 0$ with $\frac{\delta}{\eta}\to 0$ with $Q^{\eta,\delta}(x)\in \mathcal{Q}_{Fr}$ for every $\eta,\delta>0$ $\forall x\in A_j$;
\item a
{compact set $B_j$},  well contained in $A_j$ where {$Q^{\delta,\eta}$} are constant, coincide with $Q^{\eta}$, and $\operatorname{meas}(A_j\setminus B_j)\leq \delta/\eta$, provided that $\eta \gg \delta>0$;
\item a constant $C>0$ such that, for every (fixed) $\alpha>0$
\begin{equation}\label{}
	\frac{1}{2} \int_{A_j\setminus B_j}  \left| \e^{\alpha}\nabla' Q^{\eta, \delta}\right|^2 + \left|{\pt Q^{\eta, \delta}} \right|^2   
\leq C_j  \frac{\delta\eta}{m\delta^2}\eta^{-2};
\end{equation}
\item a piecewise-affine vector map
$f(x):\R^3\mapsto\R^3$ such that
  $Q^{1}(x)=\frac{\nabla f+\nabla^T f}{2}(x)$
where $Q^1$
 is the periodic extension to $\R^3$ of the tensor $Q^{\eta}$ computed for $\eta=1$.
\end{enumerate}

\end{lemma}

\begin{proof}
These are explicit constructions. 
{See proof to Theorem 4.3 \cite{cesana2018variational}}.
\end{proof}

\begin{Remark}
Lemma~\ref{lem:m3asfacts}
 revolves around a two-fold limiting process parameterised by $\delta$ and $\eta$. The first limit (in $\eta\to 0$) identifies piece-wise constant maps approximating biaxial optic states which are constant with respect to the thickness by exhibiting fine scale optic textures (item 1 and 5).
The second limit (in $\delta\to 0$) allows to smoothly interpolate such oscillating optic states ({at scale $\eta$}) 
by smooth transitions occurring on $A_j\setminus B_j$, across small boundary layers of thickness $\delta>0$ (item 2).
{
The set  $B_j\subset A_j$ introduced in item 3,
corresponds to a countable union of small disjoints sets and is the region where the mollified optic microstructure is constant.
}
In light of the  geometry of the system and the material length scales exhibited in in items 1, 2, and 3, we can estimate the error on Frank's curvature energy along boundary layers (item 4).
Finally, item 5 guarantees the existence of a microstructure that allows
both energy relaxation   and convexification of the nematic manifold.
{See also, in the language of  differential inclusions~\cite{ADMDS15}, \cite{CDPRZZ20}.}

\end{Remark}

\begin{proof}[Proof of Proposition \ref{prop:glimsupp0}]

We explicitly construct the recovery sequence in the film and in the bonding layer
\begin{equation}
	v^{\e}(x):=\left\{
	\begin{aligned}
	 &v_f^{\e}(x) & \text{ in }\Of\\
	 &v_b^{\e,{\eta}}(x) & \text{ in }\Ob\\
	 \end{aligned}\right.,
\end{equation}
The recovery sequence is three dimensional and accounts for mechanical reduction and the emergence of optic textures at two different length scales $\eta, \e$ in $\Ob$. Displacements are continuous across the interface so that $v^{\e}\in H^1(\O,\R^3)$ for every $\e.$
In the film, the displacement profile entails a vanishing shear, whilst the converging term $h^{\varepsilon}$ is introduced to satisfy optimality between the in-plane and the out-of-plane deformations.
Within the nematic bonding layer, in order to recover boundary conditions and interface continuity, a tailored microstructure is necessary to relax the optic tensor by formation of (weakly converging) sequences of micro-scale three-dimensional deformation patterns at length scale $\eta$.
{Finally, tight transition layers of size $ \delta>0$ } allow to smoothly accommodate these rapidly varying optic domains.
For the reader's convenience we split the discussion, treating film layer and bonding layer separately.
\paragraph{$\Gamma$-limsup film.}We start with the rescaled energy, defining the recovery sequence in the film
\begin{equation}\label{eqn:recsecfilm}
	v_f^{\e}(x):=\begin{system}
	\zeta'(x')- x_3\nabla'\zeta_3(x')\\
	 \zeta_3(x')+\e^2h^\e(x)\\
	 \end{system}, \quad\zeta_3\in H^2(\omega), \zeta_\alpha\in H^1(\omega).
\end{equation}
Here, we choose $h^\e \in C^\infty_c(\Of)$ such that 
{
$\pt h_\e(x)\to 
\coeftransvopt\tr(e'(u))=
\coeftransvopt  \tr \left(e'(\zeta')- \nabla'\nabla' \zeta_3 x_3\right)$}
strongly in $L^2(\Ob)$
 as $\e\to 0$ in order to satisfy optimality of transverse strains.
The associated strain components are
\begin{equation}
	e'(\ve) = e'(\zeta')-x_3 \pab \zeta_3, \quad
	\e^{-2}\ett(\ve) = \pt h_\e, \quad
	\e^{-1} \eat(\ve) = \e \pa h_\e,
\end{equation}
note the cancellation in the shear term which allows to approximate vanishing shear deformations, for $\e\to 0$.
Plugging \eqref{eqn:recsecfilm} into $J^{\e}_f$, passing to the limit using the characterisation of $h_\e$, and computing the exact integral along the vertical coordinate, leads to
\begin{multline}\label{2012242025}
	\lim_{\e\to 0}\frac{1}{2}\int_\o \int_0^1 
	\left( |e'(\zeta')|^2 - 2x_3 e'(\zeta)\nabla'\nabla' \zeta_3 + x_3^2 |\nabla'\nabla' \zeta_3|^2 + (\pt h_\e)^2 + \frac{\e^2}{2}|\nabla' h_\e|^2 \right) dx\\
	+ {\frac{1}{2}}\int_\o \int_0^1 \lamm \left( \operatorname{tr} e'(\zeta') -x_3 \Delta' \zeta_3 + \pt h_\e \right)^2 dx
	= \frac{1}{2}\int_\o 
	 \left( |e'(\zeta')|^2 - e'(\zeta')\nabla'\nabla' \zeta_3 + \frac{1}{3}|\nabla'\nabla' \zeta_3|^2 \right) dx'\\
	  + 
{\frac{1}{2}}\int_\o	 \coefopt \left(\operatorname{tr}^2 e'(\zeta') - \operatorname{tr} e'(\zeta') \Delta' \zeta_3 + \frac{1}{3}(\Delta' \zeta_3)^2 \right)dx'.
\end{multline}
This gives us the asymptotic film energy, a common contribution to both the thick-nematic and thin-nematic regimes. Here the unknown is the  displacement at the interface $\omega \times \{0\}$.

\paragraph{$\Gamma$-limsup nematic layer.} In the active layer, the strategy consists in finding an upper bound to the
two-variable integral 
$\Jepen(u,Q,0)$ (which is turn an upper bound to the functional $\Jep(u)$). 
We target a piecewise constant {$\overline Q\in L^2(\o, \mathcal Q_B)$} by constructing a recovery sequence tailored to account for the dimension reduction in the elastic regime as well as for the optical relaxation.
The latter is achieved through a martensite-like microstructure on a collection of grains 
$\{A_j\}_{j=1}^m$ where $\overline Q$ is  constant. 
With some abuse of notation, without risk of confusion, we indicate with $\overline Q$ both the piecewise constant field over $\Ob$ as well as the constant matrix over a specific $A_j$.
There, we approximate our relaxed target  biaxial optic tensor by a weakly converging oscillating sequence.
The key elements for the construction of the optic sequence draw heavily from~\cite{cesana2018variational} and are recalled in Proposition~\ref{lem:m3asfacts}.
The careful estimates of error terms and boundary layers require lengthy calculations which we omit, referring the interested reader to \cite{cesana2018variational} for explicit details. 
We treat the regime $p\in (-1, 0)$ as a particular {case} characterised by the decoupling of membrane deformations from bending modes.

On a single grain $A_j$, the recovery sequence for displacements in the nematic layer can be written as follows
\begin{equation}
v_j^{\e, \eta} (x)= u^\star(x', 0)(x_3+1) + \theta(x) w^{\e, \eta}(x), \qquad x\in A_j
\end{equation}
where $u^\star(x', 0) = (\zeta_1, \zeta_2, \zeta_3)(x', 0)$ (see~\eqref{eqn:recsecfilm}) 
{is the trace film displacements at the interface. In order to simplify the notation, we label $v^\star:=u^\star(x', 0)(x_3+1)$ the target affine displacement. In the expression above} $\theta(x)\in C^\infty(A_j): A_j\mapsto [0, 1]$ is a smooth three-dimensional cutoff function which, in each grain, is used to recover homogeneous displacements at the grain boundary.
We choose $\theta \equiv 1$ on a compact set well contained in $A_j$ at distance $\rho$ from its boundary and can always assume $|\nabla\theta|\le \rho^{-1}$.
The oscillating sequence $w^{\e, \eta}$ reads
\begin{equation}
w^{\e,\eta}:=
 f^{\e,\eta} -   \overline Q\left( {x'/\e},x_3/\e^{-p} \right)^T,
  \textrm{ with }	 f^{\e,\eta}:=
 \begin{system}
	 \eta  f_\alpha(x'/(\e \eta), x_3/(\eta\e^{-p}))\\
	 \eta \e^{-p}f_3(x'/(\eta\e ), x_3/(\eta\e^{-p}))\\
	 \end{system}
	 \label{eqn:oscillfetaeps}
\end{equation}
where $f$ is the vector field defined in Proposition~\ref{lem:m3asfacts}, properly rescaled to account for the thin film scaling.
Notice the difference in frequency of oscillations between in-plane and the out-of-plane displacements.

Observe that, by construction, $|w^{\e, \eta}|\leq \eta$ uniformly in $x$ and $\e$.
{
Indeed, this descends from the following fundamental properties: for fixed $\e>0$, 
$f^{\e,\eta} \to  \overline Q\left( {x'/\e},x_3/\e^{-p} \right)^T$ uniformly in $A_j$
and
$	\hat{\kappa}^\e(f^{\e,\eta}) -\overline{Q}=Q^{\eta} -\overline{Q} \rightharpoonup  0$
weakly in $L^2(A_j,\R^{3\times 3})$ as $\eta\to 0$, thanks to Proposition~\ref{lem:m3asfacts}
 items 1,2.
}

Furthermore, $v_b^{\e, \eta}$ matches the displacement of the film at the interface $\omega \times \{0\}$ ensuring the necessary continuity.
Recalling the definition of scaled strains introduced in~\eqref{2007092341}, we can compute $\hat{\kappa}^\e(v_b^{\e, \eta}) =\hat{\kappa}^\e(v^{\star}) +\hat{\kappa}^\e(\vartheta w^{\e, \eta}) $ term by term.
 Scaled strains of the target displacement $v^\star$ read

\smallskip{\begin{equation}\label{2101212030}
	\hat{\kappa}^\e(v^*) = 
\left( \begin{matrix}
\e e'(\zeta') & \frac{1}{2}\left( \e^{p+1}(x_3+1)\nabla' \zeta_3 + \e^{-p}\zeta' \right) \\
	\text{sym} & \zeta_3
\end{matrix} \right).
\end{equation}
}
As expected, depending on the value of $p$, either both in-plane and transverse components of displacements, or only transverse displacements contribute in the limit.
Similarly, scaled strains associated to the optic contribution read
\begin{equation}\label{2007091211}
	\hat{\kappa}^\e(\theta w^{\e, \eta})=
\left( 
\begin{matrix}
	\e \nabla'\theta \otimes_s w' & \frac{1}{2}\left( \e^{p+1}\nabla' \theta w_3 + \e^{-p}\pt\theta w'\right)\\
	\text{sym } & \pt \theta w_3
	\end{matrix} \right) + \theta \hat{\kappa}^\e(w^{\e, \eta})
\end{equation}
where the last summand is equal to $\theta (Q^{\eta} - \overline Q)$ by construction, 
{see \eqref{eqn:oscillfetaeps} and Proposition
\ref{lem:m3asfacts}, items 3 and 5.
}

We now show that the recovery sequence just built is optimal on a generic grain $A_j$
by splitting  the energy integral  in a bulk and a boundary layer contribution.
Indeed, consider the compact set $B_j\subset A_j$ introduced in Proposition~\ref{lem:m3asfacts}, item 3.
Let $B_j^{\rho}=B_j\cap \{x: \operatorname{dist}(x, \partial A_j) > \rho\}$ be the largest compact set, well contained in $A_j$, where simultaneously $\theta$ and $Q^{\eta, \delta}$ are constant. Thanks to the estimate $\operatorname{meas}(A_j\backslash B_j)\le\delta/\eta$ we have  $\operatorname{meas}(A_j\backslash B^{\rho}_j)\le\delta/\eta+\rho$.

Considering the nematic layer energy~\eqref{2009051223}, we can now compute the energy contribution of the grain $A_j$  
along the recovery sequence $(v^{\e, \eta}, Q^{\eta, \delta})$, isolating the bulk term and estimating the residual of boundary layers. 

{By making explicit the local dependence on the domain of integration of the integral functionals} we write
\begin{multline}\label{2007072232}
	  {J_b^\e}(v^{\e,\eta}_j,Q^{\eta,\delta};A_j)=
	  \underbrace{\frac{1}{2}\int_{B_j^{\rho}} |\kappa(v^{\e, \eta})-Q^{\eta, \delta}|^2 + \lamm\operatorname{tr}^2 \kappa(v^{\e, \eta})dx}_{\termA}\\
	  +\underbrace{\frac{1}{2}\int_{A_j\setminus B_j^{\rho}} 
	  \left|\kappa(v^*) 
	 +\theta \kappa(w^{\e, \eta})+
	 \left(
	\begin{matrix}
	 	\e\nabla' \theta \otimes_s (w^{\e, \eta})' & \frac{1}{2}\e^{p+1}\nabla'\theta w^{\e,\eta}_3 + \e^{-p}\pt \theta (w^{\e, \eta})'\\
	 	\text{sym} & \pt \theta w^{\e, \eta}_3
	 \end{matrix}\right)
	 -Q^{\eta, \delta} \right|^2dx}_{\termB}
	 		 \\
+ \underbrace{\frac{1}{2}\int_{A_j\setminus B_j^{\rho}}\lamm
		 \operatorname{tr}^2(\kappa(v^*)+\e \nabla'\theta \otimes_s (w^{\e, \eta})' +\pt \theta w^{\e, \eta}_3 e_3\otimes e_3)dx}_{\termC}
			 	+\underbrace{\frac{1}{2}\delta^2_\e \int_{A_j\setminus B_j^{\rho}} \left| \e^{p+1}\nabla' Q^{\eta, \delta}\right|^2 + \left|{\pt Q^{\eta, \delta}} \right|^2dx}_{\termD}.  \\
\nonumber
\end{multline}
Using some algebra, we obtain
\begin{multline}
		 \termB+\termC\leq
M \Big(
\underbrace{\int_{A_j\setminus B_j^{\rho}} 
	  	\left|\kappa(v^\star) \right|^2+
\operatorname{tr}^2(\kappa(v^\star))dx}_{\termEi}
	 +\underbrace{\int_{A_j\setminus B_j^{\rho}} 
	  \left| \theta \kappa(w^{\e, \eta}) -Q^{\eta, \delta}\right|^2dx}_{\termEii} \Big) 
	 \\+
	  M\underbrace{\int_{A_j\setminus B_j^{\rho}} 
\left[	 \left|\left(
	\begin{matrix}
	 	\e\nabla' \theta \otimes_s (w^{\e, \eta})' & \frac{1}{2}\e^{p+1}\nabla'\theta w^{\e,\eta}_3 + \e^{-p}\pt \theta (w^{\e, \eta})'\\
	 	\text{sym} & \pt \theta w^{\e, \eta}_3
	 \end{matrix}
	 \right)
	  \right|^2 
	  +
\left( 
	  \operatorname{tr}^2(\e\nabla'\theta \otimes_s (w^{\e, \eta})')+(\pt \theta w_3^{\e, \eta})^2 \right) 
\right] dx}_{\termEiii}\nonumber
\end{multline}
where $M>0$.
First,
because the integrands are bounded and 
{$\operatorname{meas}(A_j\setminus B_j^{\rho})\le \delta/\eta+\rho$}
we have the bound
\begin{equation}
	\termEi+\termEii\leq C \left( \rho+\frac{\delta}{\eta} \right). 
	\label{2007072230}
\end{equation}
For the cross term, using Schwarz's  inequality and the fact that $|w^{\e, \eta}|\leq \eta$, we have
\begin{equation}\label{2007072229}
\termEiii
	 \leq
	 C\int_{A_j\setminus B_j^{\rho}}
	 |\nabla \theta|^2|w^{\e,\eta}|^2
\leq C \frac{ \eta^2}{\rho}.
\end{equation}
Finally, in light of Proposition~\ref{lem:m3asfacts}-item 3, we have
\begin{equation}\label{2007072231}
\termD\leq C \frac{\delta^2_\e}{\delta\eta}.
\end{equation}
To reconstruct the three-dimensional limiting energy of the active layer along the recovery sequence, we first extend the construction from the single grain to the entire collection of grains, setting
\begin{equation}
	v^{\eta,\e}_b(x) = v_j^{\eta, \e}(x) \text{ on } A_j	
\end{equation}
whereby $v^{\eta,\e}_b(x) \in H^1(\Ob; \R^3)$.
Then, using the grain estimates \eqref{2007072230}, \eqref{2007072229}, and \eqref{2007072231}, we sum over the entire partition $\{A_j\}$
\begin{multline}
\limsup_{\e\to 0} 
{J}^{\e}_b(v^{\e,\eta}_b,Q^{\eta,\delta};\O_b)=\sum_{j=1}^m 
\limsup_{\e\to 0}
{J}^{\e}_b(v^{\e,\eta}_j,Q^{\eta,\delta};A_j)\\=
\limsup_{\e\to 0}     \frac{1}{2}\int_{B^{\rho}} \left(\left|\kappa^\e(v^{\star}) 
- \overline Q	  \right|^2
	+ \lamm\operatorname{tr}^2(\kappa^\e(v^{\star}) )\right)dx'
	 +C_1 \frac{\delta^2_\e}{\delta\eta}
	 +C_2 \left( \rho + \frac{\delta}{\eta}\right)
	 +C_3 \frac{ \eta^2}{\rho}\\
\le\frac{1}{2}\int_{\o} \left( |\overline Q' |^2
+ 2| \tfrac{1}{2}{(\zeta^{[p]})'}-  \overline Q|^2 + (\zeta_3-\overline Q_{33})^2
+ \lamm
\zeta_3^2\right)dx'
	 +C_4   \rho,
\label{2007072236}
\end{multline}
where $B^\rho=\cup_j B^\rho_j$ and the $C_i$'s are positive constants for $i=1,..., 4$.
In the last line we have computed the limit as $\e\to 0$ choosing a diagonal sequence $\eta=\eta(\e)$ and
$\delta=\delta(\e)$ such that
$\frac{\eta(\e)}{\e}\to 0$
and $\frac{\delta(\e)}{\e}\to 0$ as $\e\to 0 $
for fixed $\rho>0$
and extended the integration domain from $B^\rho$ to $\Ob$ owing to the non-negativity of the local (additive) energy.
We finally pass to the limit two-dimensional domain $\o$ using the columnar structure of the integration domains along the recovery sequence.
Here, we use 
$ \zeta^{[p]}_{\alpha}$
for 
$ \zeta_{\alpha}^{[p]}\equiv  \zeta_{\alpha}$
if $p=0$
and 
$ \zeta^{[p]}_{\alpha}\equiv
0$
if $p\in (-1,0).$
Because $\rho$ is fixed and arbitrary the last contribution may be made arbitrarily small.
Finally, we are able to integrate over $x_3$ and read the results separately.
{Below, $\eta, \delta$ stand for $\eta(\e),\delta(\e)$.}

\begin{equation}\label{2012242050}
\limsup_{\e\to 0} 
J_b^{\e}(v^{\e,\eta},Q^{\eta,\delta};\O_b)\leq
 	\frac{1}{2}\int_{\o}\left( |\overline Q'|^2 + (2(\overline Qe_3)'|^2 + (\zeta_3-\overline Q_{33})^2+ \lamm\zeta_3^2\right)dx'
, \quad \text{ if } p\in (-1, 0)
 	\end{equation} 
 	and
\begin{equation}\label{2012251810}
\limsup_{\e\to 0} 
J_b^{\e}(v^{\e,\eta},Q^{\eta,\delta};\O_b)\leq
 	\frac{1}{2}\int_{\o} \left(|\overline Q'|^2 + 2|\frac{1}{2}\zeta'- (\overline Qe_3)'|^2 + (\zeta_3-\overline Q_{33})^2+ \lamm\zeta_3^2\right)dx'
, \quad \text{ if } p = 0.
 	 \end{equation}

{
Now, 
we can replace the piecewise constant $\overline{Q}$ first with a general $\overline{Q}\in L^2(\o,\mathcal{Q}_B)$,
in the right hand side of expression above thanks to the continuity of the energy and density properties of order tensors (Proposition~3 of~\cite{cesana2011nematic}).
Second, in place of the general  $\overline{Q}\in L^2(\o,\mathcal{Q}_B)$, } we choose the argmin $\overline{\overline{Q}}\in
L^2(\o,\mathcal{Q}_B)$ 
of the right-hand side 
(respectively, \eqref{2012242050} and \eqref{2012251810}).
The latter is unique owing to the convexity and compactness of $\mathcal{Q}_B$.
Using the characterisation $\int_{\o} \left(   |\overline{\overline{Q'}} |^2
+ 2| \frac{1}{2}{\zeta^{[p]}_{\alpha}}-  \overline{\overline{Q}}_{\alpha 3}|^2 + (\zeta_3-\overline{\overline{Q}}_{33})^2\right)dx'=
 \int_{\o} \operatorname{dist}^2(\mathsf K((\zeta^{[p]}_1,\zeta^{[p]}_2),\zeta_3), \mathcal{Q}_B) dx' $
and summing up film and nematic layer contributions, we obtain
\begin{multline}\label{2012241920}
	\Gamma\hbox{-}\limsup_{\e\to 0}
\Jep(u)
\leq\limsup_{\e\to 0}\left( \inf_{\overline Q\in L^2(\o,\mathcal{Q}_B)}
J_b^{\e}(v^{\e, \eta}_b,\overline Q ) + J_f^{\e}(v^{\e, \eta}) \right)
\leq\limsup_{\e\to 0} 
\left( J_b^{\e}(v^{\e, \eta},Q^{\eta, \delta}) + J_f^{\e}(v^{\e, \eta}) \right)  \\
\leq
\frac{1}{2}\int_\o 
\left[	 |e'(\zeta')|^2 - e'(\zeta')\nabla'\nabla' \zeta_3 + \frac{1}{3}|\nabla'\nabla' \zeta_3|^2 + 
	 \coefopt \left(\operatorname{tr}^2 e'(\zeta') - \operatorname{tr} e'(\zeta') \Delta' \zeta_3 + \frac{1}{3}(\Delta' \zeta_3)^2 \right) \right]dx'
\\
+\frac{1}{2}
 \int_{\o} \left[\operatorname{dist}^2\left(\mathsf K \left((\zeta^{[p]})',\zeta_3\right), \mathcal{Q}_B\right) + \lamm\zeta_3^2\right]dx'.
\end{multline}

\end{proof}

\vspace{1em}
\begin{multicols}{2}
\noindent \textbf{Pierluigi Cesana} \\
Institute of Mathematics for Industry, \\
Kyushu University\\
744 Motooka, Nishi-ku\\
Fukuoka 819-0395, Japan

\noindent e-mail: \url{cesana@math.kyushu-u.ac.jp}

\columnbreak

\noindent \textbf{Andrés A. León Baldelli} \\
Institute of Mechanical Sciences\\
and Industrial Applications (IMSIA)\\
CNRS UMR 9219, Palaiseau\\
France\\
\noindent e-mail: \url{leon.baldelli@cnrs.fr}

\end{multicols}

\end{document}